\documentclass[11pt]{article}
\usepackage[margin=.91in]{geometry}

\usepackage{graphicx} 
\usepackage{amsmath} 
\usepackage{amssymb}  
\usepackage{amsthm}

\usepackage{color}
\allowdisplaybreaks[1]

\usepackage{algorithm}
\usepackage{algpseudocode}
\usepackage{hyperref} 
\usepackage{enumerate}

\theoremstyle{plain}  
\newtheorem{theorem}{Theorem}[section]
\newtheorem{lem}[theorem]{Lemma}

\newtheorem{cor}[theorem]{Corollary}
\newtheorem{defn}[theorem]{Definition}

\theoremstyle{definition}

\theoremstyle{remark}
\newtheorem{example}{Example}[section]
\newtheorem{remark}{Remark}[section]

\def\0{{\bf 0}}
\def\1{{\bf 1}}

\def \Snn{ \mathbb{S}^{n \times n}}
\def \nulls{ \mathcal{N}}
\def \cols{ \mathcal{C}}
\def \bmat{\left[\begin{matrix}}
\def \emat{\end{matrix}\right]}
\def \bvec{\left(\begin{matrix}}
\def \evec{\end{matrix}\right)}

\def \xy1vec{\left[\begin{matrix}x\\y\\1\end{matrix}\right]}
\def \QED{\begin{flushright}\Halmos\end{flushright}\end{proof}}
\def \defeq{\mathrel{\mathop{:}}=}
\def \Tr{\mathrm{Tr}}
\def \xbar{\bar{x}}
\def \ybar{\bar{y}}
\def \grad{\nabla}
\def \gx{\grad p(x)}
\def \Hx{\Hess p(x)}

\def \gxb{\grad p(\xbar)}
\def \Hxb{\Hess p(\xbar)}
\def \gp3d{\grad p_3(d)}
\def \Hess{\nabla^2}
\def \R{\mathbb{R}}
\def \Rn{\R^n}
\def \beq{\begin{equation}}
\def \eeq{\end{equation}}
\def \baeq{\begin{equation*}\begin{aligned}}
\def \eaeq{\end{aligned}\end{equation*}}
\newcommand{\baeql}[1]{\begin{equation}\label{#1}\begin{aligned}}
\def \eaeql{\end{aligned}\end{equation}}

\def \otn{\{1, \ldots, n\}}

\long\def\edit#1{{\color{black}#1}}
\title{\LARGE \bf Complexity Aspects of Local Minima\\ and Related Notions\thanks{This work was supported partially by an AFOSR MURI award, the DARPA Young Faculty Award, the Princeton SEAS Innovation Award, the NSF CAREER Award, the Google Faculty Award, and the Sloan Fellowship.}}
\author{Amir Ali Ahmadi\thanks{Department of Operations Research and Financial Engineering, Princeton University; \texttt{aaa@princeton.edu}}\hspace{1mm} and Jeffrey Zhang\thanks{Department of Mathematical Sciences, Carnegie Mellon University; \texttt{jeffz@cmu.edu}}}
\begin{document}
\date{}
\maketitle
\vspace{-1.8em}
\begin{abstract}
	\noindent 
	We consider the notions of (i) critical points,  (ii) second-order points,  (iii) local minima, and (iv) strict local minima for multivariate polynomials. For each type of point, and as a function of the degree of the polynomial, we study the complexity of deciding (1) if a given point is of that type, and (2) if a polynomial has a point of that type. Our results characterize the complexity of these two questions for all degrees left open by prior literature. Our main contributions reveal that many of these questions turn out to be tractable for cubic polynomials. In particular, we present an efficiently-checkable necessary and sufficient condition for local minimality of a point for a cubic polynomial. We also show that a local minimum of a cubic polynomial can be efficiently found by solving semidefinite programs of size linear in the number of variables. By contrast, we show that it is strongly NP-hard to decide if a cubic polynomial has a critical point. We also prove that the set of second-order points of any cubic polynomial is a spectrahedron, and conversely that any spectrahedron is the projection of the set of second-order points of a cubic polynomial. In our final section, we briefly present a potential application of finding local minima of cubic polynomials to the design of a third-order Newton method.

\end{abstract}
\section{Introduction}\label{Sec: Introduction}

We are concerned in this paper with algorithmic questions around the following four types of points associated with a sufficiently smooth function $f:\Rn \to \R$:
\begin{enumerate}[(i)]
    \item a \emph{critical point}, i.e., a point $x$ where the gradient $\grad f(x)$ is zero,
    \item a \emph{second-order point}, i.e., a point $x$ where $\grad f(x) = 0$ and the Hessian $\Hess f(x)$ is positive semidefinite (psd), i.e. has nonnegative eigenvalues,
    \item a \emph{local minimum}, i.e., a point $x$ for which there exists a scalar $\epsilon > 0$ such that $f(x) \le f(y)$ for all $y$ with $\|y - x\| \le \epsilon$,
    \item a \emph{strict local minimum}, i.e., a point $x$ for which there exists a scalar $\epsilon > 0$ such that $f(x)<f(y)$ for all $y \ne x$ with $\|y - x\| \le \epsilon$.
\end{enumerate}
We note the following straightforward implications between (i)-(iv):
\begin{center}
    strict local minimum $\Rightarrow$ local minimum $\Rightarrow$ second-order point $\Rightarrow$ critical point.
\end{center}

Notions (i)-(iv) appear ubiquitously in nonconvex continuous optimization as surrogates for global minima. This is because it is well understood that finding a global minimum of $f$ is in general an intractable problem.

In this paper, with regard to each of the four notions above, we study the complexity of answering the following questions:
\begin{enumerate}[{\bf Q1:}]
    \item Given a function $f: \Rn \to \R$ and a point $x \in \Rn$, is $x$ of a given type (i)-(iv)?
    \item Given a function $f: \Rn \to \R$, does $f$ have a point of a given type (i)-(iv) (and if so, can one be found efficiently)?
\end{enumerate}

Note that a priori there are no complexity implications between these two questions. For example, an algorithm for verifying that a given point is a local minimum does not necessarily provide instructions on how one would find a local minimum. Conversely, even if local minimality of a given point cannot always be efficiently certified, that does not rule out the existence of algorithms that can efficiently find particular local minima that are easy to certify; see e.g. Question 3 of~\cite{pardalos1992open}. Thus, in general, these two questions need to be studied separately.

The functions $f$ for which we study Q1 and Q2 are (multivariate) polynomials. Polynomial functions appear throughout optimization theory either as exact models of objective functions or as approximations thereof. For example, many optimization algorithms involve minimizing Taylor expansions of more complicated functions as a subroutine. As is well known, polynomials can approximate continuous functions arbitrarily well over compact sets. This makes them a particularly suitable candidate for studying local notions such as (i)-(iv). In addition to these representation reasons, since polynomial functions of a given degree are finitely parameterized, they allow for a convenient setting for a formal study of complexity questions. For example, one can study the complexity of Q1 and Q2 in the Turing model of computation, where the size of a given instance is determined by the number of bits required to write down the coefficients of the polynomial (and, in the case of Q1, the entries of the point $x$), which are taken to be rational numbers. For the purposes of analyzing the complexity of these two questions for polynomial functions, we consider the relevant setting in applications where the \edit{degree\footnote{\edit{In this paper, by a degree-$d$ polynomial, we mean a polynomial whose monomials have degree at most $d$. All our complexity results hold for this convention, as well as for the convention which requires a degree-$d$ polynomial to have at least one monomial of degree $d$ with a nonzero coefficient.}}} of the polynomial is fixed and its number of variables increases. We are interested in the existence or non-existence of efficient algorithms for solving these questions in this setting, as established theory (e.g. quantifier elimination theory~\cite{tarski1951decision, seidenberg1954new}) already yields exponential-time algorithms for them.

Let us first comment on the complexity of Q1 and Q2 for some simple and classical cases. For Q1, checking whether a given point is a critical point of a polynomial function (of any degree) can trivially be done in polynomial time simply by evaluating the gradient at that point. To check that a given point is a second-order point, one can additionally compute the Hessian matrix at that point and check that it is positive semidefinite. This can be done in polynomial time, e.g., by performing Gaussian pivot steps along the main diagonal of the matrix~\cite[Section 1.3.1]{murty1988linear} or by computing its characteristic polynomial and checking that the signs of its coefficients alternate~\cite[p. 403]{horn2012matrix}. Since for affine or quadratic polynomials, any second-order point is a local minimum, the only remaining case for Q1 is that of strict local minima. Affine polynomials never have strict local minima, making the question uninteresting. A point is a strict local minimum of a quadratic polynomial if and only if it is a critical point and the associated Hessian matrix is positive definite (pd), i.e., has positive eigenvalues. The latter property can be checked in polynomial time, for example by computing the leading principal minors of the Hessian and checking that they are all positive. As for Q2, the affine case is again uninteresting since there is a critical point (which will also be a second-order point and a local minimum) if and only if the coefficients of all degree-one monomials are zero. For quadratic polynomials, since the entries of the gradient are affine, searching for critical points can be done in polynomial time by solving a linear system. A candidate critcal point will be a second-order point (and a local minimum) if and only if the Hessian is psd, and a strict local minimum if and only if the Hessian is pd.

Other than the aforementioned cases, the only prior result in the literature that we are aware of is due to Murty and Kabadi~\cite{murty1987some}, which settles the complexity of Q1 for degree-4 polynomials. Our contribution in this paper is to settle the complexity of the remaining cases for both Q1 and Q2. A summary of the results is presented in Table \ref{Table: Complexity Checking} and Table \ref{Table: Complexity Existence}. Entries denoted by ``P'' indicate that the problem can be solved in polynomial time. The notation ``SDP'' indicates that the problem of interest can be reduced to solving either one or polynomially-many semidefinite programs (SDP) whose sizes are polynomial in the size of the input. (In fact, the reduction also goes in the other direction for second-order points and local minima; see Theorems~\ref{Thm: SOP SDPF} and \ref{Thm: LM SDPSF}.) Finally, we recall that a strong NP-hardness result implies that the problem of interest remains NP-hard even if the size (i.e. bit length) of the coefficients of the polynomial is $O(\log(n))$, where $n$ is the number of variables. Therefore, unless P=NP, even a pseudo-polynomial time algorithm (i.e., an algorithm whose running time is polynomial in the magnitude of the coefficients, but not necessarily their bit length) cannot exist for the indicated problems in these tables. See \cite{garey2002computers} or \cite[Section 2]{ahmadi2019complexity} for more details on the distinction between weakly and strongly NP-hard problems.

\begin{table}[H]
    \centering
    \begin{tabular}{c|c|c|c|c}
    {\bf Q1: \quad property vs. degree} & $1$ & $2$ & $3$ & $\ge 4$ \\\hline
    Critical point & P & P & P & P\\\hline
    Second-order point & P & P & P & P\\\hline
    Local minimum & P & P & P & strongly NP-hard \cite{murty1987some}\footnotemark\addtocounter{footnote}{-1}\\
    & & & (Theorem \ref{Thm: Checking Min Poly Time}) & \\\hline
    Strict local minimum & P & P & P & strongly NP-hard \cite{murty1987some}\footnotemark\\
    & & & (Corollary \ref{Cor: Check Strict Local Min}) &
    \end{tabular}
    \caption{Complexity of deciding whether a given point is of a certain type, based on the degree of the polynomial. Entries without a reference are classical.}
    \label{Table: Complexity Checking}
\end{table}
\footnotetext{The proof in~\cite{murty1987some} is based on a reduction from the ``matrix copositivity'' problem. However, \cite{murty1987some} only shows that this problem (and thus deciding if a quartic polynomial has a local minimum) is weakly NP-hard, since the reduction to matrix copositivity there is from the weakly NP-hard problem of Subset Sum. Nonetheless, their result can be strengthened by observing that testing matrix copositivity is in fact strongly NP-hard. This claim is implicit, e.g., in~\cite[Corollary 2.4]{de2002approximation}. The NP-hardness of testing whether a point is a strict local minimum of a quartic polynomial is not explicitly stated in~\cite{murty1987some}, though it follows in the weak sense from the weak NP-hardness of Problem 8 of~\cite{murty1987some}. Again, with some work, this can be strengthened to a strong NP-hardness result.}
\begin{table}[H]
    \centering
    \begin{tabular}{c|c|c|c|c}
    {\bf Q2: \quad property vs. degree} & $1$ & $2$ & $3$ & $\ge 4$ \\\hline
    Critical point & P & P & strongly NP-hard & strongly NP-hard \\
    & & & (Theorem \ref{Thm: Critical point cubic NP-hard}) & (Theorem \ref{Thm: Critical point cubic NP-hard})\\\hline
    Second-order point & P & P & SDP & strongly NP-hard\\
    & & & (Corollary \ref{Cor: Complete Cubic SDP SOP}) & (Theorem \ref{Thm: SOP quartic NP-hard})\\\hline
    Local minimum & P & P & SDP & strongly NP-hard \\
    & & & (Algorithm \ref{Alg: Complete Cubic SDP}) & (Theorem \ref{Thm: LM quartic NP-hard})\footnotemark\addtocounter{footnote}{-1}\\\hline
    Strict local minimum & P & P & SDP & strongly NP-hard\\
    & & & (Algorithm \ref{Alg: Complete Cubic SDP}, Remark \ref{Rem: Find SLM}) & (Theorem \ref{Thm: LM quartic NP-hard})\footnotemark
    \end{tabular}
    \caption{Complexity of deciding whether a polynomial has a point of a certain type, based on the degree of the polynomial. Entries without a reference are classical.}
    \label{Table: Complexity Existence}
\end{table}
\footnotetext{The proof of Theorem \ref{Thm: LM quartic NP-hard} will appear in an upcoming paper by the authors, as a corollary of it answers a question originally posed in \cite{pardalos1992open}; see Section \ref{Sec: NP-hardness results}.}

The majority of the technical work in this paper is spent on the case of cubic polynomials. It is somewhat surprising that many of the problems of interest to us are tractable for cubics, especially the search for local minima. This is in contrast to the intractability of other interesting problems related to cubic polynomials, e.g., minimizing them over the unit sphere~\cite{nesterov2003random}, or checking their convexity over a box~\cite{ahmadi2019complexity}. It is also interesting to note that second-order points of cubic polynomials are easier to find than their critical points, despite being a more restrictive type of point. This shows that the right approach to finding second-order points involves bypassing the search for critical points as an initial step.

\subsection{Organization and Main Contributions of the Paper}

Section \ref{Sec: NP-hardness results} covers the NP-hardness results from Table~\ref{Table: Complexity Existence}. The remainder of the paper is devoted to our results on cubic polynomials, which fills in the remaining entries of Tables~\ref{Table: Complexity Checking} and \ref{Table: Complexity Existence}. In Section~\ref{Sec: Local Min}, we give a characterization of local minima of cubic polynomials (Theorem~\ref{Thm: TOC}) and show that it can be checked in polynomial time (Theorem \ref{Thm: Checking Min Poly Time}). In Section~\ref{Sec: Geometry}, we give some geometric facts about local minima of cubic polynomials. For example, we show that the set of local minima of a cubic polynomial $p$ is convex (Theorem~\ref{Thm: WLM Convex}), and we relate this set to the second-order points of $p$
and to the set of minima of $p$ over points where $\Hess p$ is positive semidefinite (Theorem~\ref{Thm: Closure} and Theorem~\ref{Thm: Local Minima SOP}). In Section~\ref{SSec: Cubic Spectrahedra}, we show that the interior of any spectrahedron is the projection of the local minima of some cubic polynomial (Theorem~\ref{Thm: Cubic Spectrahedron}). In Section~\ref{Sec: Complexity}, we use this result to show that deciding if a cubic polynomial has a local minimum or a second-order point is at least as hard as some semidefinite feasibility problems.

In Section~\ref{Sec: Finding Local Min}, we start from a ``sum of squares'' approach to finding second-order points of a cubic polynomial (Theorem~\ref{Thm: cubic sdp} and Theorem~\ref{Thm: solution recovery sdp}), and build upon it (Section~\ref{SSec: Simplification}) to arrive at an efficient semidefinite representation of these points (Corollary~\ref{Cor: Complete Cubic SDP SOP}). This also leads to an algorithm for finding local minima of cubic polynomials by solving polynomially-many SDPs of polynomial size (Algorithm~\ref{Alg: Complete Cubic SDP}). In Section~\ref{Sec: Conclusions}, we take preliminary steps towards some interesting future research directions, such as the design of an unregularized third-order Newton method that would use as a subroutine our algorithm for finding local minima of cubic polynomials (Section \ref{SSec: Cubic Newton}).

\subsection{Preliminaries and Notation}\label{SSec: Preliminaries}

We review some standard facts about local minina; more preliminaries specific to cubic polynomials appear in Section~\ref{SSec: Cubic Preliminaries}. Three well-known optimality conditions in unconstrained optimization are the \emph{first-order necessary condition} (FONC), the \emph{second-order necessary condition} (SONC), and the \emph{second-order sufficient condition} (SOSC). Respectively, they are that the gradient at any local minimum is zero, the Hessian at any local minimum is psd, and that any critical point at which the Hessian is positive definite is a strict local minimum. A vector $d \in \Rn$ is said to be a \emph{descent direction} for a function $p: \Rn \to \R$ at a point $\xbar \in \Rn$ if there exists a scalar $\epsilon > 0$ such that $p(\xbar + \alpha d) < p(\xbar)$ for all $\alpha \in (0,\epsilon)$. Existence of a descent direction at a point clearly implies that the point is not a local minimum. However, in general, the lack of a descent direction at a point does not imply that the point is a local minimum (see, e.g., Example \ref{Ex: no local min}).

Next, we establish some basic notation which will be used throughout the paper. We denote the set of $n\times n$ real symmetric matrices by $\Snn$. For a matrix $M\in\Snn$, the notation $M \succeq 0$ denotes that $M$ is positive semidefinite, $M \succ 0$ denotes that it is positive definite, and $\Tr(M)$ denotes its trace, i.e. the sum of its diagonal entries. For a matrix $M$, the notation $\nulls(M)$ denotes its null space, and $\cols(M)$ denotes its column space. All vectors are taken to be column vectors. For two vectors $x$ and $y$, the notation $(x,y)$ denotes the vector $\bvec x \\ y \evec$. The notation $0_n$ denotes the vector of length $n$ containing only zeros. The notation $e_i$ denotes the $i$-th coordinate vector, i.e., the vector with a one in its $i$-th entry and zeros everywhere else.

\section{NP-hardness Results}\label{Sec: NP-hardness results}

In this section, we present reductions that show our NP-hardness results from Tables \ref{Table: Complexity Checking} and \ref{Table: Complexity Existence}. For concreteness, we construct these reductions from the (simple) MAXCUT problem, though our proof can work with any NP-hard problem that can be encoded by quadratic equations with ``small enough'' coefficients. Recall that in the (simple) MAXCUT problem, we are given as input an undirected and unweighted graph $G$ on $n$ vertices and an integer $k \le \edit{\frac{n(n-1)}{2}}$. We are then asked whether there is a cut in $G$ of size $k$, i.e. a partition of the vertices into two sets $S_1$ and $S_2$ such that the number of edges with one endpoint in $S_1$ and one endpoint in $S_2$ is equal to $k$. It is well known that the (simple) MAXCUT problem is strongly NP-hard~\cite{garey2002computers}.
	
If we denote the adjacency matrix of $G$ by $E \in \Snn$, it is straightforward to see that $G$ has a cut of size $k$ if and only if the following system of quadratic equations is feasible:
	
	\begin{equation}\label{Eq: MAXCUT QP}
    \begin{aligned}
	q_0(x) &\defeq \frac{1}{4}\sum_{i=1}^n \sum_{j=1}^n E_{ij}(1-x_ix_j) - k = 0,\\
	q_i(x) &\defeq x_i^2 - 1 = 0, i = 1, \ldots, n.
	\eaeql
	Indeed, the second set of constraints enforces each variable $x_i$ to be $-1$ or $1$, and any $x \in \{-1,1\}^n$ encodes a cut in $G$ by assigning vertices with $x_i = 1$ to one side of the partition, and those with $x_i=-1$ to the other. Observe that with this encoding, $x_ix_j$ equals $1$ whenever the two vertices $i$ and $j$ are on the same side and $-1$ otherwise. The size of the cut is therefore given by $\frac{1}{4}\sum_{i=1}^n \sum_{j=1}^n E_{ij}(1-x_ix_j)$, noting that every edge is counted twice.

\begin{theorem}\label{Thm: Critical point cubic NP-hard}
	It is strongly NP-hard to decide whether a polynomial $p: \Rn \to \R$ of degree greater than or equal to three has a critical point.
\end{theorem}

\begin{proof}
	\edit{We prove this statement for a degree-3 polynomial, which also trivially proves it for polynomials of degree greater than 3.\footnote{\edit{If one desires the polynomial in our reduction to have a nonzero term of degree $d \ge 4$, then this can be done, for example, by introducing another variable $v$, and adding the term $v^d$ to our construction. The same claim applies to the proof of Theorem~\ref{Thm: SOP quartic NP-hard}.}}}
	
	Given an instance of the (simple) MAXCUT problem with a graph on $n$ vertices, let the quadratic polynomials $q_0, \ldots, q_n$ be as in (\ref{Eq: MAXCUT QP}), and consider the following \edit{degree-3} polynomial in \edit{$2n+1$} variables \edit{$(x_1, \ldots, x_n, y_0, y_1, \ldots, y_n)$:
	$$p(x,y) = \sum_{i=0}^{n} y_iq_i(x).$$}
	Note that all coefficients of this polynomial take $O(\log(n))$ bits to write down. We show that \edit{$p(x,y)$} has a critical point if and only if the quadratic system $q_0(x) = 0, \ldots, q_{n}(x) = 0$ is feasible. Observe that the gradient of $p$ is given by
	
	\edit{$$\bvec \\ \frac{\partial p}{\partial x}\\\\\hline\\ \frac{\partial p}{\partial y} \\\\\evec = \bvec \sum_{i=0}^n y_i\frac{\partial q_i}{\partial x_1}(x) \\ \vdots \\ \sum_{i=0}^n y_i\frac{\partial q_i}{\partial x_n}(x)\\ \hline q_0(x) \\ \vdots \\ q_n(x)\evec.$$}
	
	If $\xbar \in \Rn$ is a solution to (\ref{Eq: MAXCUT QP}), then the point \edit{$(\xbar, 0_{n+1})$} is a critical point of $p$. Conversely, if \edit{$(\xbar, \ybar)$} is a critical point of $p$, then, since \edit{$\frac{\partial p}{\partial y}(\xbar,\ybar) = 0$}, $\xbar$ must be a solution to (\ref{Eq: MAXCUT QP}).
\end{proof}

\begin{theorem}\label{Thm: SOP quartic NP-hard}
	It is strongly NP-hard to decide whether a polynomial $p: \Rn \to \R$ of degree greater than or equal to four has a second-order point.
\end{theorem}

\begin{proof}
    \edit{We prove this statement for degree-4 polynomials, which also trivially proves it for polynomials of degree greater than 4.}

	Given an instance of the (simple) MAXCUT problem with a graph on $n$ vertices, let the quadratic polynomials $q_0, \ldots, q_n$ be as in (\ref{Eq: MAXCUT QP}), and consider the following \edit{degree-4} polynomial in \edit{$3n+2$} variables \edit{$(x_1, \ldots, x_n, y_0, y_1, \ldots, y_n, z_0, z_1, \ldots, z_n)$}:
	\edit{$$p(x,y,z) = \sum_{i=0}^{n} \left(y_i^2 q_i(x) - z_i^2 q_i(x)\right).$$}
	Note that all coefficients of this polynomial take $O(\log(n))$ bits to write down. We show that \edit{$p(x,y,z)$} has a second-order point if and only if the quadratic system $q_0(x) = 0, \ldots, q_{n}(x) = 0$ is feasible. 
	
	Observe that $\frac{\partial^2 p}{\partial y^2}$ is an $(n+1) \times (n+1)$ diagonal matrix with $2q_0(x), \ldots, 2q_n(x)$ on its diagonal. Similarly, $\frac{\partial^2 p}{\partial z^2}$ is an $(n+1) \times (n+1)$ diagonal matrix with $-2q_0(x), \ldots, -2q_n(x)$ on its diagonal. Suppose first that \edit{$(\xbar, \ybar, \bar z)$} is a second-order point of $p$. Since \edit{$\Hess p(\xbar, \ybar, \bar z) \succeq 0$}, and since $\frac{\partial^2 p}{\partial y^2}$ and $\frac{\partial^2 p}{\partial z^2}$ are both principal submatrices of $\Hess p$, it must be that $q_0(\xbar) = 0, \ldots, q_n(\xbar) = 0$.
	
	Now suppose that $\xbar \in \Rn$ is a solution to (\ref{Eq: MAXCUT QP}). We show that \edit{$(\xbar, 0_{n+1}, 0_{n+1})$} is a second-order point of $p$. Note that $\frac{\partial p}{\partial x}$ is quadratic in $y$ and $z$, $\frac{\partial p}{\partial y}$ is linear in $y$, \edit{and} $\frac{\partial p}{\partial z}$ is linear in $z$. 
	Thus \edit{$(\xbar, 0_{n+1}, 0_{n+1})$} is a critical point of $p$. Now observe that the entries of $\frac{\partial^2 p}{\partial x^2}$ are quadratic in $y$ and $z$ or are zero, the entries of $\frac{\partial^2 p}{\partial x \partial y}$ are linear in $y$ or are zero, the entries of $\frac{\partial^2 p}{\partial x \partial z}$ are linear in $z$ or are zero,
	\edit{$\frac{\partial^2 p}{\partial y^2}(\xbar, 0_{n+1}, 0_{n+1})$ and $\frac{\partial^2 p}{\partial z^2}(\xbar, 0_{n+1}, 0_{n+1})$} are both zero, and all other entries of $\Hess p$ are zero. Thus \edit{$\Hess p(\xbar, 0_{n+1}, 0_{n+1}) = 0$}, and we conclude that \edit{$(\xbar, 0_{n+1}, 0_{n+1})$} is a second-order point of $p$.
\end{proof}

The remaining two NP-hardness results from Table \ref{Table: Complexity Existence} are stated next, but proven in an upcoming paper by the authors in~\cite{Local_Min_QP}. The reason we have decided to present this result separately is that a corollary of it answers a question of Pardalos and Vavasis on existence of an efficient algorithm for finding a local minimum of a quadratic function over a polytope. This question appeared in 1992 on a list of seven open problems in complexity theory for numerical optimization~\cite{pardalos1992open} and is answered negatively in~\cite{Local_Min_QP}.

\begin{theorem}[\cite{Local_Min_QP}]\label{Thm: LM quartic NP-hard}
    It is strongly NP-hard to decide whether a polynomial $p: \Rn \to \R$ of degree greater than or equal to four has a local minimum. The same statement holds for testing existence of a strict local minimum.
\end{theorem}

\section{Checking Local Minimality of a Point for a Cubic Polynomial}\label{Sec: Local Min}

As the reader can observe from Tables~\ref{Table: Complexity Checking} and \ref{Table: Complexity Existence} from Section~\ref{Sec: Introduction}, the remaining entries all have to do with the case of cubic polynomials. To answer these questions about cubics, we start in this section by showing that the problem of deciding if a given point is a local minimum (or a strict local minimum) of a cubic polynomial is polynomial-time solvable. This answers the remaining cases in Table~\ref{Table: Complexity Checking}. 
We first make certain observations about cubic polynomials that will be used throughout the paper. 

\subsection{Preliminaries on Cubic Polynomials}\label{SSec: Cubic Preliminaries}

It is easy to observe that a univariate cubic polynomial has either no local minima, exactly one local minimum (which is strict), or infinitely many non-strict local minima (in the case that the polynomial is constant). Further observe that if a point $\xbar \in \Rn$ is a (strict) local minimum of a function $p: \Rn \to \R$, then for any fixed point $\ybar \in \Rn$, \edit{with $\ybar \ne \xbar$}, the restriction of $p$ to the line going through $\xbar$ and $\ybar$ ---i.e. the univariate function $q(\alpha) \defeq p(\xbar + \alpha (\ybar - \xbar))$---has a (strict) local minimum at $\alpha = 0$. Since the restriction of a multivariate cubic polynomial to any line is a univariate polynomial of degree at most three, the previous two facts imply that (i) if a cubic polynomial has a strict local minimum, then it must be the only local minimum (strict or non-strict), and that (ii) if a cubic polynomial has multiple local minima, then the polynomial must be constant on the line connecting any two of these (necessarily non-strict) local minima.

Observe that for any cubic polynomial $p$, the error term of the second-order Taylor expansion is the cubic homogeneous component of $p$. More formally, for any point $\xbar \in \Rn$ and direction $v \in \Rn$,
\beq \label{Eq: cubic taylor}
p(\xbar + \lambda v) = p_3(v)\lambda^3 + \frac{1}{2}v^T \Hxb v \lambda^2 + \gxb^Tv\lambda + p(\xbar),
\eeq 
where $p_3$ is the collection of terms of $p$ of degree exactly 3.

Note that the Hessian of any cubic $n$-variate polynomial is an affine matrix of the form $\sum_{i=1}^n x_iH_i + Q$, where $H_i$ and $Q$ are all $n \times n$ symmetric matrices and the $H_i$ satisfy

\begin{equation}\label{Eq: Valid Hessian}
(H_i)_{jk} = (H_j)_{ik} = (H_k)_{ij}
\end{equation}
for any $i,j,k \in \{1, \ldots, n\}$. This is because an $n \times n$ symmetric matrix $A(x) \defeq A(x_1, \ldots, x_n)$ is a valid Hessian matrix if and only if $\frac{\partial}{\partial x_i} A_{jk}(x) = \frac{\partial}{\partial x_j} A_{ik}(x) = \frac{\partial}{\partial x_k} A_{ij}(x)$ for all $i,j,k \in \{1, \ldots, n\}$. If $\sum_{i=1}^n x_iH_i + Q$ is a Hessian matrix, then the cubic polynomial which gives rise to it is of the form

\beq \label{Eq: Cubic Poly Form} \frac{1}{6}\sum_{i=1}^n x^T x_iH_ix + \frac{1}{2}x^TQx + b^Tx + c.\eeq
In this paper, it is sometimes convenient for us to parametrize a cubic polynomial in the above form. As the scalar term in (\ref{Eq: Cubic Poly Form}) is irrelevant for deciding local minimality or finding local minima, in the remainder of this paper, we take $c=0$ without loss of generality. Observe that the gradient of the polynomial in (\ref{Eq: Cubic Poly Form}) is $\frac{1}{2}\sum_{i=1}^n x_iH_ix + Qx + b$, or equivalently a vector whose $i$-th entry is $\frac{1}{2}x^TH_ix + e_i^TQx + b_i$.

\subsection{Local Minimality of a Point for a Cubic Polynomial}

In this section, we give a characterization of local minima of cubic polynomials and show that this characterization can be checked in polynomial time. Recall that we use the notation $p_3$ to denote the cubic homogeneous component of a cubic polynomial $p$, and $\nulls(M)$ (resp. $\cols(M)$) to denote the null space (resp. column space) of a matrix $M$.

\begin{theorem}\label{Thm: TOC}
	A point $\xbar \in \Rn$ is a local minimum of a cubic polynomial $p: \Rn \to \R$ if and only if the following three conditions hold:
	\begin{itemize}
		\item $\grad p(\xbar) = 0,$
		\item $\Hess p(\xbar) \succeq 0,$
		\item $\gp3d = 0, \forall d \in \nulls(\Hxb).$
	\end{itemize}
\end{theorem}

Note that the first two conditions are the well-known FONC and SONC. Throughout the paper, we refer to the third condition as the \emph{third-order condition (TOC)} for optimality. This condition is requiring the gradient of the cubic homogeneous component of $p$ to vanish on the null space of the Hessian of $p$ at $\xbar$. We remark that the FONC, SONC, and TOC together are in general neither sufficient nor necessary for a point to be a local minimum of a polynomial of degree higher than three. The first claim is trivial (consider, e.g., $p(x) = x^5$ at $x = 0$); for the second claim see Example \ref{Ex: TOC not necessary}.

\begin{remark}\label{Rem: TONC}
It is straightforward to see that any local minimum $\xbar$ of a cubic polynomial $p$ satisfies a condition similar to the TOC, that $p_3(d) = 0, \forall d \in \nulls(\Hxb)$. Indeed, if $\xbar$ is a second-order point and $d \in \nulls(\Hxb)$, then Equation (\ref{Eq: cubic taylor}) gives $p(\xbar + \lambda d) = p_3(d)\lambda^3 + p(\xbar)$. Hence, if $p_3(d)$ is nonzero, then either $d$ or $-d$ is a descent direction for $p$ at $\xbar$, and so $\xbar$ cannot be a local minimum. This observation was made in~\cite{anandkumar2016efficient} for three-times differentiable functions, and is referred to as the ``third-order necessary condition'' (TONC) for optimality. Note that because $p_3$ is homogeneous of degree three, from Euler's theorem for homogeneous functions we have $3 p_3(x) = x^T\grad p_3(x)$. We can then see that $\grad p_3(d) = 0 \Rightarrow p_3(d) = 0$, and therefore the TOC is a stronger condition than the TONC. Indeed, the FONC, SONC, and TONC together are not sufficient for local optimality of a point for a cubic polynomial; see Example \ref{Ex: no local min}. Intuitively, this is because the FONC, SONC, and TONC together avoid existence of a descent direction for cubic polynomials, but as the proof of Theorem \ref{Thm: TOC} will show, existence of a ``descent parabola'' must also be avoided.
\end{remark}

We will need the following fact from linear algebra for the proof of Theorem \ref{Thm: TOC}.

\begin{lem}\label{Lem: Smallest Nonzero Eigenvalue}
	Let $M \in \Snn$ be a symmetric positive semidefinite matrix and denote its smallest positive eigenvalue by $\lambda_+$. Then if $z \in \cols(M)$ and $\|z\| = 1, z^TMz \ge \lambda_+$.
\end{lem}
\begin{proof}
	Suppose $M$ has eigenvalues $\lambda_1 \ge \lambda_2 \ge \cdots \ge \lambda_k > \lambda_{k+1} = \cdots = \lambda_n = 0$ (so $\lambda_+ = \lambda_k$). Let $v_1, \ldots, v_n$ be a set of corresponding mutually orthogonal unit-norm eigenvectors of $M$. Observe that any $z \in \cols(M)$ can be written as $z = \sum_{i=1}^n \alpha_iv_i$, for some scalars $\alpha_i$ with $\alpha_i = 0$ for $i =k+1, \ldots, n$. This is because the column space is orthogonal to the null space, and the eigenvectors corresponding to zero eigenvalues span the null space.
	
	Since $v_1, \ldots, v_k$ are mutually orthogonal unit vectors, we have
	$$z^TMz = \left(\sum_{i=1}^k \alpha_iv_i\right)^T \left(\sum_{i=1}^k \lambda_iv_iv_i^T\right) \left(\sum_{i=1}^k \alpha_iv_i\right) = \sum_{i=1}^k \alpha_i^2\lambda_i v_i^Tv_iv_i^Tv_i = \sum_{i=1}^k \alpha_i^2 \lambda_i,$$
	and
	$$1 = \|z\|^2 = \left(\sum_{i=1}^k \alpha_iv_i\right)^T\left(\sum_{i=1}^k \alpha_iv_i\right) = \sum_{i=1}^k \alpha_i^2v_i^Tv_i = \sum_{i=1}^k \alpha_i^2.$$
	These two equations combined imply that $z^TMz \ge \lambda_k = \lambda_+$.
\end{proof}

\begin{proof}[Proof (of Theorem \ref{Thm: TOC})]
	As any local minimum must satisfy the FONC and SONC, it suffices to show that a second-order point is a local minimum for a cubic polynomial if and only if it also satisfies the TOC.
	
	We first observe that for any second-order point $\xbar$, scalars $\alpha$ and $\beta$, and vectors $d \in \nulls(\Hxb)$ and $z \in \Rn$, the following identity holds:
	\begin{equation}\label{Eq: Taylor NC decomposition}
    \begin{aligned}
	p(\xbar + \alpha d+\beta z)	&=  p_3(\alpha d+\beta z) + \frac{1}{2}(\alpha d+\beta z)^T\Hxb(\alpha d+\beta z) + p(\xbar)\\
	&= \beta^3p_3(z) + \frac{\beta^2}{2}z^T\Hess p_3(\alpha d)z + \beta\grad p_3(\alpha d)^Tz + p_3(\alpha d) + \frac{\beta^2}{2} z^T \Hxb z + p(\xbar)\\
	&= \beta^3p_3(z) + \frac{\alpha \beta^2}{2}z^T\Hess p_3(d)z + \alpha^2 \beta\grad p_3(d)^Tz + \alpha^3p_3(d) + \frac{\beta^2}{2} z^T \Hxb z + p(\xbar).\\
	\eaeql
	The first equality follows from (\ref{Eq: cubic taylor}) and the FONC. The second equality follows from the Taylor expansion of $p_3(\alpha d+\beta z)$ around $\alpha d$ and using the fact that $d \in \nulls(\Hxb)$. The last equality follows from homogeneity of $p_3$.
		
	\noindent\textbf{(second-order point) + TOC $\Rightarrow$ local minimum:}
	
	Let $\xbar$ be any second-order point at which the TOC holds. Note that any vector $v \in \Rn$ can be written as $\alpha d + \beta z$ for some scalars $\alpha$ and $\beta$, and unit vectors $d \in \nulls(\Hxb)$ and $z \in \cols(\Hxb)$ \edit{(which are all unique up to sign)}. Since from the TOC we have $\grad p_3(d) = 0$ (which also implies that $p_3(d) = 0$, as seen e.g. by Euler's theorem for homogeneous functions mentioned above), the identity in (\ref{Eq: Taylor NC decomposition}) reduces to \beq
	\label{Eq: pxbar - p}
	p(\xbar + v) - p(\xbar) =\beta^2 \left(\beta p_3(z) + \frac{\alpha}{2} z^T \Hess p_3(d)z + \frac{1}{2} z^T \Hxb z\right).
	\eeq	

	Let $\lambda > 0$ be the smallest nonzero eigenvalue of $\Hxb$. From Lemma \ref{Lem: Smallest Nonzero Eigenvalue} we have that $z^T \Hxb z \ge~\lambda$. Thus, if $\alpha$ and $\beta$ satisfy 
	\beq\label{Eq: Min Radius}| \alpha| + | \beta| \le \lambda \cdot \left(\underset{\|z\|=1, \|d\|=1}{\max} \max \{z^T \Hess p_3(d)z, 2p_3(z)\}\right)^{-1},\eeq the expression on the right-hand side of (\ref{Eq: pxbar - p}) is nonnegative. Because the set $\{\|z\|=1\} \cap \{\|d\|=1\}$ is compact and $p_3$ is continuous and odd, the quantity $$\gamma \defeq \underset{\|z\|=1, \|d\|=1}{\max} \max \{z^T \Hess p_3(d)z, 2p_3(z)\}$$ is finite and nonnegative, and thus $\lambda/\gamma$ is positive (or potentially $+\infty$). Finally, note that for any $v \in \Rn$ such that $\|v\| \le \edit{\frac{\sqrt{2}}{2}}(\lambda/\gamma)$, the corresponding $\alpha$ and $\beta$ satisfy (\ref{Eq: Min Radius}), and thus $p(\xbar + v) - p(\xbar) \ge 0$ as desired.\\\\
	\textbf{Local minimum $\Rightarrow$ TOC:}
	
	Note that if $\xbar$ is a local minimum, then we must have $p_3(d) = 0$ whenever $d \in \nulls(\Hxb)$ (see Remark \ref{Rem: TONC}). We also assume that $p_3$ is not the zero polynomial, as then the TOC would be automatically satisfied. 
	
	Now suppose for the sake of contradiction that there exists a vector $\hat{d} \in \nulls(\Hxb)$ such that $\grad p_3(\hat{d}) \ne 0$. Consider the sequence of points given by
	\begin{equation}\label{Eq: Descent Parabola}
	\hat{x}_i \defeq \xbar + \alpha_i \hat{d} + \beta_i z,
	\eeq
	where
	$$z = -\frac{\grad p_3(\hat{d})}{\|\grad p_3(\hat{d})\|}, \alpha_i = \frac{1}{i}\sqrt{\frac{z^T \Hxb z}{|\grad p_3(\hat{d})^Tz|}}, \beta_i = \frac{1}{i^2}.$$
	Observe that $\hat{x}_i \to \xbar$ as $i \to \infty$.
	From (\ref{Eq: Taylor NC decomposition}), we have
	$$p(\xbar + \alpha_i \hat{d}+\beta_i z) - p(\xbar)
	= p_3(z)\beta_i^3 + \frac{1}{2}z^T\Hess p_3(\hat{d})z \alpha_i \beta_i^2 + \grad p_3(\hat{d})^Tz \alpha_i^2 \beta_i + \frac{1}{2} z^T \Hxb z \beta_i^2.$$
	Note that because $\alpha_i \propto \sqrt{\beta_i}$, the third and fourth terms of the right-hand side of the above expression will be the dominant terms as $i \to \infty$. For our choices of $\alpha_i$ and $\beta_i$, the sum of these two dominant terms simplifies to $-\frac{1}{2i^4}z^T\Hxb z$. Observe that for any $w \in \nulls(\Hxb)$ and any $\alpha \in \R, p_3(\hat{d}+\alpha w) = 0$. Since the gradient of $p_3$ is orthogonal to its level sets, we must then have $\grad p_3(\hat{d})^Tw = 0$ for any $w \in \nulls(\Hxb)$. Thus, $\grad p_3(\hat{d})$ is in the orthogonal complement of $\nulls(\Hxb)$, i.e. in $\cols (\Hxb)$, and hence $z^T \Hxb z > 0$. Thus, for any sufficiently large $i$, $p(\hat{x}_i) < p(\xbar)$, and so $\xbar$ is not a local minimum.
	
\end{proof}

\begin{remark} Note that the points $\hat{x}_i$ constructed in (\ref{Eq: Descent Parabola}) trace a parabola as $i$ ranges from $-\infty$ to $+\infty$. Thus as a corollary of the proof of Theorem~\ref{Thm: TOC}, we see that if a point $\xbar \in \Rn$ is not a local minimum of a cubic polynomial $p: \Rn \to \R$, then there must exist a ``descent parabola'' that certifies that; i.e. a parabola $q(t): \R \to \Rn$ and a scalar $\bar{\alpha}$ satisfying $q(0) = \xbar$ and $p(q(\alpha)) < p(\xbar)$ for all $\alpha \in (0, \bar{\alpha})$.
\end{remark}

Theorem~\ref{Thm: TOC} gives rise to the following algorithmic result.


\begin{theorem}\label{Thm: Checking Min Poly Time}
	Local minimality of a point $\xbar \in \Rn$ for a cubic polynomial $p: \Rn \to \R$ can be checked in polynomial time.
\end{theorem}

\begin{proof}
	In view of Theorem \ref{Thm: TOC}, we show that the FONC, SONC, and TOC can be checked in polynomial time (in the Turing model of computation). Checking that the gradient of $p$ vanishes at $\xbar$ and that the Hessian at $\xbar$ is positive semidefinite can be done in polynomial time as explained in Section~\ref{Sec: Introduction}. We give the following polynomial-time algorithm for checking the TOC:
	
	\begin{algorithm}[H]
		\caption{Algorithm for checking the TOC.}\label{Alg: Check Local Min}
		\begin{algorithmic}[1]
			\State {\bf Input:} Coefficients of a cubic polynomial $p: \Rn \to \R$, a point $\xbar \in \Rn$
			\State Compute $\Hxb$
			\State Compute a rational basis $\{v_1, \ldots, v_k\}$ for the null space of $\Hxb$
			\State Check if coefficients of $g(\lambda) \defeq \grad p_3 (\sum_{i=1}^k \lambda_i v_i)$ are all zero
    		\State \quad \texttt{{\bf if}} YES
    			\State \quad \quad $\xbar$ is a local minimum of $p$
    		\State \quad \texttt{{\bf if}} NO
    			\State \quad\quad $\xbar$ is a not local minimum of $p$
		\end{algorithmic}
	\end{algorithm}
	
	Note that the entries of the function $g: \R^k \to \Rn$ that appears in this algorithm are homogeneous quadratic polynomials in $\lambda \defeq (\lambda_1, \ldots, \lambda_k)$, where $k$ is the dimension of $\nulls(\Hxb)$. For the TOC to hold, $g$ must be zero for all $\lambda \in \R^k$, which happens if and only if all coefficients of every entry of $g$ are zero.
	
	A rational basis for the null space of a symmetric matrix can be computed in polynomial time, for example through the Bareiss algorithm \cite{bareiss1968sylvester}. For completeness, we give a less efficient but also polynomial-time algorithm which solves a series of linear systems. The first linear system finds a nonzero vector $v_1 \in \Rn$ such that $\Hxb^T v_1 = 0$. The successive linear systems solve for nonzero vectors $v_i \in \Rn$ such that $\Hxb^T v_i = 0, v_j^Tv_i = 0, \forall j = 1, \ldots, i-1$. To ensure nonzero solutions, some entry of the vector is fixed to 1, and if the system is infeasible, the next entry is fixed to 1 and the system is re-solved. Once the only feasible vector is the zero vector, the basis is complete.
	
	The next step is to compute the coefficients of $g$. To do this, one can first compute the coefficients of $\grad p_3$. There are $n \times {n+1 \choose 2}$ coefficients to compute, and each is a coefficient of $p_3$, multiplied by 1, 2, or 3. If the $m$-th entry of $\grad p_3$ is given by $\sum_{i=1}^n \sum_{j\ge i}^n c_{ij}x_ix_j$, then the $m$-th entry of $g$ is equal to $g_m(\lambda) = \sum_{a=1}^n \sum_{b=1}^n (\sum_{i=1}^n \sum_{j\ge i}^n c_{ij}(v_a)_i(v_b)_j)\lambda_a\lambda_b$, where the vectors $\{v_i\}$ are our rational basis for $\nulls(\Hxb)$. Observe that $g_m$ is a polynomial in $\lambda$ whose coefficients can be computed with a polynomial number of additions and multiplications over polynomially-sized scalars, and thus checking if all these coefficients are zero for every $m$ can be done in polynomial time.
	
	\end{proof}
	
	Let us end this subsection by also giving an efficient characterization of strict local minima of cubic polynomials.
	
	\begin{cor}\label{Cor: Strict Local Min}
		A point $\xbar \in \Rn$ is a strict local minimum of a cubic polynomial $p: \Rn \to \R$ if and only if
		\begin{itemize}
			\item $\gxb = 0,$
			\item $\Hxb \succ 0.$
		\end{itemize}
	\end{cor}	
	
	\begin{proof}
	    The fact that these two conditions are sufficient for local minimality is immediate from the SOSC. To show the converse, in view of the FONC, we only need to show that positive definiteness of the Hessian is necessary. Suppose for the sake of contradiction that for some nonzero vector $d \in \Rn$, we have $d^T \Hxb d = 0$ (note that in view of the SONC, we cannot have $d^T \Hxb d< 0$). From (\ref{Eq: cubic taylor}), we have $p(\xbar+\alpha d) = p(\xbar) + p_3(d)\alpha^3$. Hence, $\alpha = 0$ is not a strict local minimum of the univariate polynomial $p(\xbar + \alpha d)$, and so $\xbar$ is not a strict local minimum of $p$.
	\end{proof}
	
	\begin{cor}\label{Cor: Check Strict Local Min}
	Strict local optimality of a point $\xbar \in \Rn$ for a cubic polynomial $p: \Rn \to \R$ can be checked in polynomial time.
	\end{cor}
	\begin{proof}
		This follows from the characterization in Corollary~\ref{Cor: Strict Local Min}. Checking the FONC is straightforward as before. As explained in Section~\ref{Sec: Introduction}, to check that $\Hxb$ is positive definite, one can equivalently check that all $n$ leading principal minors of $\Hxb$ are positive. This procedure takes polynomial time since determinants can be computed in polynomial time.
	\end{proof}

	\subsection{Examples}\label{SSec: Local Min Examples}
	
	We give a few illustrative examples regarding the application and context of Theorem \ref{Thm: TOC}.
	
	\begin{figure}[h]
	\includegraphics[height=.28\textheight,keepaspectratio]{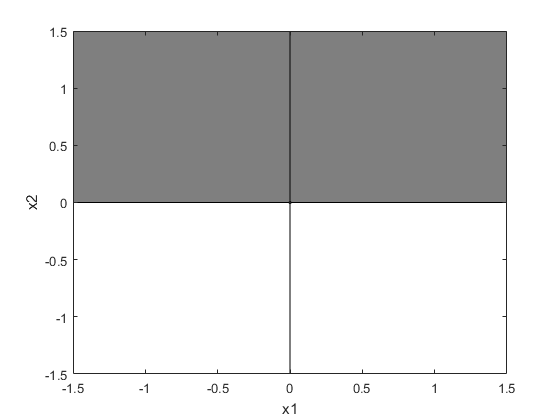}
	\includegraphics[height=.28\textheight,keepaspectratio]{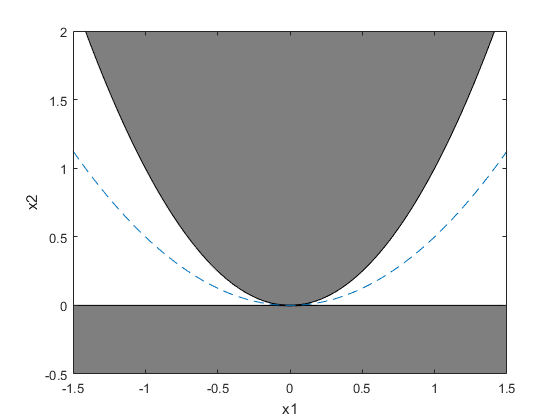}
	\caption{Contour plots of $x_1^2x_2$ (left) and $x_2^2 - x_1^2x_2$ (right) from Examples \ref{Ex: local min} and \ref{Ex: no local min}. The polynomials are zero on the black lines, positive on the gray regions, and negative on the white regions. The dashed line in the right-side figure denotes a descent parabola at the origin.}
	\label{Fig: Local Min contours}
	\end{figure}
	
	\begin{example}\label{Ex: local min}
		{\bf A cubic polynomial with local minima}
	
	Consider the polynomial $p(x_1,x_2) = x_1^2x_2$. By inspection (see Figure \ref{Fig: Local Min contours}), one can see that points of the type $\{(x_1,x_2)\ |\ x_1 = 0, x_2 > 0\}$ are local minima of $p$, as $p$ is nonnegative when $x_2 > 0$, zero whenever $x_1 = 0$, and positive whenever $x_2 > 0$ and $x_1 \ne 0$. As a sanity check, we use Theorem~\ref{Thm: TOC} to verify that the point $(0,1)$ is a local minimum of $p$ (the same reasoning applies to all other local minima).
	
	Through straightforward computation, we find
	
	$$\gx = \bvec 2x_1x_2 \\ x_1^2 \evec, \grad p_3(x) = \bvec 2x_1x_2 \\ x_1^2 \evec, \Hx = \bmat 2x_2 & 2x_1 \\ 2x_1 & 0 \emat.$$
	We can see that the FONC and SONC are satisfied at $(0,1)$.	The null space of $\Hess p(0,1)$ is spanned by $(0,1)$. We have
	$$\grad p_3\left(\alpha \bvec 0 \\ 1 \evec\right) = \bvec 2(0)(\alpha) \\ (0)^2 \evec = 0,$$
	which shows that the TOC is satisfied, verifying that $(0,1)$ is a local minimum of $p$.
	
	One can also verify that $\{(x_1,x_2)\ |\ x_1 = 0, x_2 > 0\}$ are the only local minima. Indeed, the critical points of $p$ are those where $x_1 = 0$, and the second-order points are those where $x_1 = 0$ and $x_2 \ge 0$. To see that $(0,0)$ is not a local minimum, observe that $(1,1) \in \nulls (\Hess p(0,0))$, but $\grad p_3(1,1) = (2,1) \ne 0$, and thus the TOC is violated.
	\end{example}
	
	\begin{example}\label{Ex: no local min}
		{\bf A cubic polynomial with no local minima}
	
	We use Theorem~\ref{Thm: TOC} to show that the polynomial $p(x_1,x_2) = x_2^2 - x_1^2x_2$ has no local minima. We have
	
	$$\gx = \bvec -2x_1x_2 \\ 2x_2 - x_1^2 \evec, \grad p_3(x) = \bvec -2x_1x_2 \\ -x_1^2 \evec, \Hx = \bmat -2x_2 & -2x_1 \\ -2x_1 & 2 \emat.$$
	Observe that $(0,0)$ is the only second-order point of $p$. The null space of $\Hess p(0,0)$ is spanned by $(1,0)$. We have
	
	$$\grad p_3\left(\alpha \bvec 1 \\ 0 \evec\right) = \bvec -2(\alpha)(0) \\ -(\alpha)^2 \evec = \bvec 0 \\ -\alpha^2 \evec \ne 0,$$
	which shows that the TOC is violated, and hence $(0,0)$ is not a local minimum. Note that the TONC is in fact satisfied at $(0,0)$, since $p_3(\alpha, 0) = 0$ for any scalar $\alpha$.
	
It is also interesting to observe that there are no descent directions for $p$ at $(0,0)$ (this is implied, e.g., by satisfaction of the TONC, along with the FONC and SONC). However, we can use the proof of Theorem \ref{Thm: TOC} to compute a descent parabola, thereby more explicitly demonstrating that $(0,0)$ is not a local minimum. The column space of $\Hess p(0,0)$ is spanned by $(0,1)$. Then, following the proof of Theorem \ref{Thm: TOC} with $z = (0,1)$ and $\hat{d} = (1,0)$, we have $z^T \Hess p(0,0) z = 2$ and $|\grad p_3(\hat{d})^T z| = 1$. The parabola prescribed is then the set $\{(x_1, x_2)\ |\ x_2 = \frac{1}{2} x_1^2\}$. Indeed, one can now verify that except at $(0,0)$, $p$ is negative on the entire parabola; see the dashed line in Figure \ref{Fig: Local Min contours}.
	\end{example}
	
	\begin{example}\label{Ex: TOC not necessary}
	{\bf A quartic polynomial with a local minimum that does not satisfy the TOC}
	
	
	We show in this example that for polynomials of degree higher than three, the TOC is not a necessary condition for local minimality. Consider the polynomial $p(x_1,x_2) = 2x_1^4 + 2x_1^2x_2 + x_2^2$. The point $(0,0)$ is a local minimum, as $p(0,0) = 0$ and $p(x_1,x_2) = x_1^4 + (x_1^2 + x_2)^2$ is nonnegative. However, the Hessian of $p$ at $(0,0)$ is
	$$\Hess p(0,0) = \bmat 0 & 0 \\ 0 & 2 \emat,$$
	which has a null space spanned by $(1,0)$. We observe that $\grad p_3(x_1,x_2) = \bmat 4x_1x_2 \\ 2x_1^2 \emat$ does not vanish on this null space, as it evaluates, for example, to $(0,2)$ at $(1,0)$.
	\end{example}
	
	\section{On the Geometry of Local Minima of Cubic Polynomials}\label{Sec: Geometry}
	
	We have shown that deciding local minimality of a given point for a cubic polynomial is a polynomial-time solvable problem. We now turn our attention to the remaining unresolved entries in Table~\ref{Table: Complexity Existence} from Section~\ref{Sec: Introduction}, which are on the problems of deciding whether a cubic polynomial has a second-order point, a local minimum, or a strict local minimum. In Sections~\ref{Sec: Complexity} and \ref{Sec: Finding Local Min}, we will show that these problem can all be reduced to semidefinite programs of tractable size. In the current section, we present a number of geometric results about local minima and second-order points of cubic polynomials which are used in those sections, but are possibly of independent interest.
	
	For the remainder of this paper, we use the notation $SO_p$ to denote the set of second-order points of a polynomial $p$, $LM_p$ to denote the set of its local minima, and $\bar{S}$ to denote the closure of a set $S$.
	
	\subsection{Convexity of the Set of Local Minima}\label{SSec: Local Min Convex}
	
	We begin by showing that for any cubic polynomial $p$, the set $LM_p$ is convex. We go through two lemmas; the first is a simple algebraic observation, and the second contains information about some critical points. Recall that the Hessian of a cubic polynomial $p$ written in the form of (\ref{Eq: Cubic Poly Form}) is given by $\sum_{i=1}^n x_iH_i + Q$. Furthermore, its gradient is given by $\frac{1}{2}\sum_{i=1}^n x_iH_ix + Qx + b$, or equivalently a vector whose $i$-th entry is $x^TH_ix + e_i^TQx + b_i$.
	
	\begin{lem}\label{Lem: Hessian switch}
		Let $H_1, \ldots, H_n \subseteq \Snn$ satisfy (\ref{Eq: Valid Hessian}). Then for any two vectors $y, z \in \Rn$, $$\left(\sum_{i=1}^n y_iH_i\right)z = \left(\sum_{i=1}^n z_iH_i\right)y.$$
	\end{lem}
	\begin{proof}
		Observe that for any index $k \in \{1, \ldots, n\}$, we have
		\baeq
		\left(\left(\sum_{i=1}^n y_iH_i\right)z\right)_k &= \sum_{i=1}^n \sum_{j=1}^n (H_i)_{kj}y_iz_j\\
		&= \sum_{i=1}^n \sum_{j=1}^n (H_j)_{ki}y_iz_j\\
		&= \left(\left(\sum_{j=1}^n z_jH_j\right)y\right)_k,
		\eaeq
		where the second equality follows from (\ref{Eq: Valid Hessian}).
	\end{proof}
	
	\begin{lem}\label{Lem: Gradient Invariance}
		Let $\xbar \in \Rn$ be a local minimum of a cubic polynomial $p: \Rn \to \R$, and let $d~\in~\nulls(\Hxb)$. Then for any scalar $\alpha$, $\xbar +\alpha d$ is a critical point of $p$.
	\end{lem}
	\begin{proof}
		Let $p$ be given in our canonical form as $\frac{1}{6}\sum_{i=1}^n x^Tx_iH_ix + \frac{1}{2}x^TQx + b^Tx$. We have
		
		\baeq
		\grad p(\xbar+ \alpha d) &= \left(\frac{1}{2} \sum_{i=1}^n \xbar_iH_i + \alpha d_iH_i\right)(\xbar+\alpha d) + Q(\xbar+\alpha d) + b\\
		&= \left(\frac{1}{2}\sum_{i=1}^n \xbar_iH_i\right)\xbar + Q\xbar + b\\
		&+ \frac{1}{2} \sum_{i=1}^n \alpha d_iH_i\xbar + \frac{1}{2}\sum_{i=1}^n \alpha \xbar_iH_id + \alpha Qd\\
		&+ \frac{\alpha^2}{2}\sum_{i=1}^{n} d_iH_id\\
		&=\grad p(\xbar) + \alpha \Hess p(\xbar)d + \alpha^2\grad p_3(d)\\
		&= 0 + 0 + 0 = 0,\eaeq
		where the third equality follows form Lemma \ref{Lem: Hessian switch}, and the last follows from the FONC and TOC.
	\end{proof}
	
	\begin{theorem}\label{Thm: WLM Convex}
		The set of local minima of any cubic polynomial is convex.
	\end{theorem}
	\begin{proof}
		If for some cubic polynomial $p$, the set $LM_p$ of its local minima is empty or a singleton, the claim is trivially established. Otherwise, let $\xbar, \ybar \in LM_p$ with $\xbar \ne \ybar$. Consider any convex combination $z \defeq \xbar + \alpha(\ybar-\xbar)$, where $\alpha \in (0,1)$. We show that $z$ satisfies the FONC, SONC, and TOC, and therefore by Theorem~\ref{Thm: TOC}, $z \in LM_p$.
		
		Note from (\ref{Eq: cubic taylor}) that the restriction of $p$ to the line passing through $\xbar$ and $\ybar$ is
		$$p(\xbar + \alpha (\ybar - \xbar)) = p_3(\ybar - \xbar)\alpha^3 + \frac{1}{2}(\ybar - \xbar)^T \Hxb (\ybar - \xbar) \alpha^2 + \gxb^T(\ybar - \xbar)\alpha + p(\xbar).$$
		Since this univariate cubic polynomial has two local minima at $\alpha = 0$ and $\alpha = 1$, it must be constant. In particular, the coefficient of $\alpha^2$ must be zero, and because $\Hxb$ is psd, that implies $\ybar-\xbar \in \nulls(\Hxb)$.  Hence, by Lemma \ref{Lem: Gradient Invariance}, the FONC holds at $z$. To show the SONC and TOC at $z$, note that because $\Hess p(x)$ is affine in $x$, $\Hess p(z)$ can be written as a convex combination of $\Hxb$ and $\Hess p(\ybar)$, both of which are psd. The SONC is then immediate. To see why the TOC holds, recall that the null space of the sum of two psd matrices is the intersection of the null spaces of the summand matrices. Thus $\nulls(\Hess p(z)) \subseteq \nulls(\Hxb)$, and the TOC is satisfied.
	\end{proof}
	
	As a demonstration of Theorem \ref{Thm: WLM Convex}, Figure \ref{Fig: Critical nonconvex} shows the critical points and the local minima of the cubic polynomial \beq\label{Eq: Nonconvex CPs}x_1^3 + 3x_1^2x_2 + 3x_1x_2^2 + x_2^3 - 3x_1 - 3x_2.\eeq Note that the critical points form a nonconvex set, while the local minima constitute a convex subset of the critical points.
	
	\begin{figure}[h]
	\centering
	\includegraphics[height=.26\textheight,keepaspectratio]{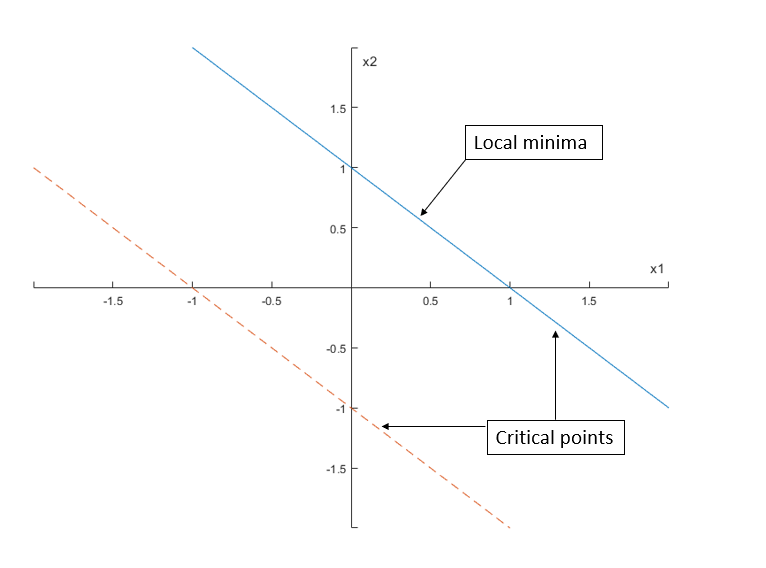}
	\caption{The critical points of the polynomial (\ref{Eq: Nonconvex CPs}). One can verify that the set of critical points is $\{(x_1, x_2)\ |\ (x_1 + x_2)^2 = 1\}$, and that the set of local minima is $\{(x_1, x_2)\ |\ x_1 + x_2 = 1\}$. The points on the dashed line are local maxima.}
	\label{Fig: Critical nonconvex}
	\end{figure}
	
	
	Unlike the above example, $LM_p$ (or even $\overline{LM_p}$ as $LM_p$ is in general not closed) may not be a polyhedral\footnote{Recall that a \emph{polyhedron} is a set defined by finitely many affine inequalities.} set for cubic polynomials. For instance, the polynomial \beq\label{Eq: Circle LM}p(x_1,x_2,x_3,x_4) = -x_1x_3^2 + x_1x_4^2 + 2x_2x_3x_4 + x_3^2 + x_4^2,\eeq has $LM_p = \{x\in \R^4 \ |\ x_1^2+x_2^2 < 1, x_3 = x_4 = 0\}$ (see Figure \ref{Fig: Local Min Spectrahedron}). This is in contrast to quadratic polynomials, whose local minima always form a polyhedral set. We show in Theorem \ref{Thm: WLM Closure Spectrahedron}, however, that $\overline{LM_p}$ is always a spectrahedron\footnote{Recall that a \emph{spectrahedron} is a set of the type $S = \{x \in \Rn|A_0 + \sum_{i=1}^n x_iA_i \succeq 0\}$, where $A_0, \ldots A_n$ are symmetric matrices of some size $m \times m$~\cite{vinzant2014spectrahedron}.}. We first need the following lemma.
	
	\begin{figure}[H]
	\centering
	\includegraphics[height=.26\textheight,keepaspectratio]{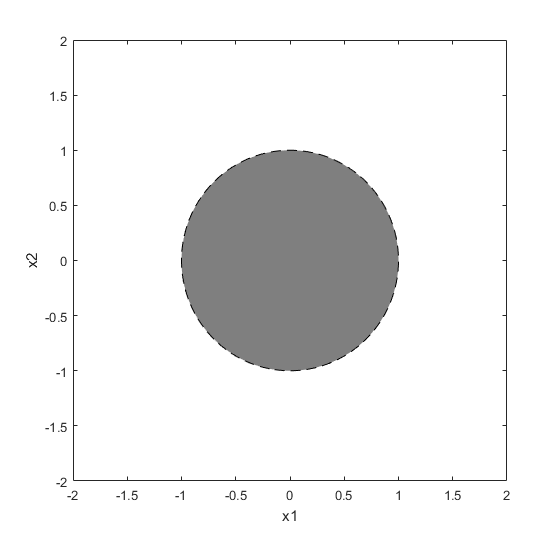}
	\caption{The projection of the set of local minima of the polynomial in (\ref{Eq: Circle LM}) onto the $x_1$ and $x_2$ variables. This example shows that $\overline{LM_p}$ is not always a polyhedral set.}
	\label{Fig: Local Min Spectrahedron}
	\end{figure}
	
	\begin{lem}\label{Lem: constant null}
			For any cubic polynomial $p: \Rn \to \R$, suppose $\xbar \in \Rn$ and $\ybar \in \Rn$ satisfy
			\begin{itemize}
				\item $\xbar \in SO_p$,
				\item $\Hess p(\ybar) \succeq 0$,
				\item $p(\xbar) = p(\ybar)$.
			\end{itemize}
			 Then $p(\xbar + \alpha (\ybar - \xbar)) = p(\xbar)$ for any scalar $\alpha$, and $\ybar - \xbar \in \nulls (\Hxb)$.
		\end{lem}
		
		 Note in particular that this lemma applies if $\ybar$ is simply a second-order point, since $p$ must take the same value at any two second-order points. This is because any non-constant univariate cubic polynomial can have at most one second-order point.
	
		\begin{proof}
			Consider the Taylor expansion of $p$ around $\xbar$ in the direction $\ybar - \xbar$ (see (\ref{Eq: cubic taylor})):
			$$q(\alpha) \defeq p(\xbar + \alpha (\ybar - \xbar)) = p_3(\ybar - \xbar)\alpha^3 + \frac{1}{2}(\ybar - \xbar)^T \Hxb (\ybar - \xbar) \alpha^2 + \gxb^T(\ybar - \xbar)\alpha + p(\xbar).$$
			Note that $q$ is a univariate cubic polynomial which has a second-order point at $\alpha = 0$. It is straightforward to see that if a univariate cubic polynomial is not constant and has a second-order point, then any other point which takes the same function value as the second-order point must have a negative second derivative. As this is not the case for $q$ (in view of $\alpha = 0$ and $\alpha = 1$), $q$ must be constant, i.e., $p(\xbar + \alpha (\ybar - \xbar)) = p(\xbar)$ for any $\alpha$. Now observe that for $p(\xbar + \alpha (\ybar - \xbar))$ to be constant, we must have $(\ybar - \xbar)^T \Hxb (\ybar - \xbar) = 0$. As $\Hxb \succeq 0$, we have $\ybar - \xbar \in \nulls (\Hxb)$.
		\end{proof}
	
	\begin{theorem}\label{Thm: WLM Closure Spectrahedron}
		For a cubic polynomial $p:\Rn \to \R, \overline{LM_p}$ is a spectrahedron.
	\end{theorem}
	\begin{proof}
	    If $LM_p$ is empty, the claim is trivial. Otherwise, let $\xbar \in LM_p$. We show that $\overline{LM_p}$ is given by the spectrahedron
	    \beq\label{Eq: WLM Closure Spectrahedron}
	    M \defeq \{x \in \Rn\ |\ \Hx \succeq 0, \Hxb(x - \xbar) =0\}.
	    \eeq
	    First consider any $\ybar \in LM_p$. From the SONC we know that $\Hess p(\ybar) \succeq 0$ and from Lemma \ref{Lem: constant null}, we know that $\ybar - \xbar \in \nulls(\Hxb)$. Thus $\ybar \in M$. Since $M$ is closed, we get that $\overline{LM_p} \subseteq M$.
	    
	    Now consider any $\ybar \in M$. By the definition of $M$, $\ybar$ satisfies the SONC, and by Lemma \ref{Lem: Gradient Invariance}, it also satisfies the FONC. Since for any scalar $\alpha \in (0,1)$, $\Hess p(\xbar + \alpha (\ybar - \xbar))$ is a convex combination of the two psd matrices $\Hxb$ and $\Hess p(\ybar)$, we have $\nulls(\Hess p(\xbar + \alpha (\ybar - \xbar))) \subseteq \Hxb$ and thus $\xbar + \alpha (\ybar - \xbar)$ satisfies the TOC (since $\xbar$ does). Thus $\ybar$ can be written as the limit of local minima of $p$ (e.g. $\{\xbar + \alpha (\ybar - \xbar)\}$ as $\alpha \to 1$).
	\end{proof}

	\begin{remark}
	We will soon show that for a cubic polynomial $p$, if $LM_p$ is nonempty, then $\overline{LM_p} = SO_p$ (see Theorem \ref{Thm: Closure}). In Section \ref{Sec: Finding Local Min}, we will give other representations of $SO_p$, which in contrast to the representation in (\ref{Eq: WLM Closure Spectrahedron}), do not rely on access to or even existence of a local minimum.
	\end{remark}

	\subsection{Local Minima and Solutions to a ``Convex'' Problem}\label{SSec: convex pop}

	In Section \ref{Sec: Finding Local Min}, we present an SDP-based approach for finding local minima of cubic polynomials. (We note again that the SDP representation in (\ref{Eq: WLM Closure Spectrahedron}) is useless for this purpose as it already assumes access to a local minimum.) Many common approaches for computing local minima of twice-differentiable functions involve first finding critical points of the function, and then checking whether they satisfy second-order conditions. However, as discussed in the introduction and in Section~\ref{Sec: NP-hardness results}, such approaches are unlikely to be effective for cubic polynomials as critical points of these functions are in fact NP-hard to find (see Theorem~\ref{Thm: Critical point cubic NP-hard}). Interestingly, however, we show in Section~\ref{Sec: Finding Local Min} that by bypassing the search for critical points, one can directly find second-order points and local minima of cubic polynomials by solving semidefinite programs of tractable size. The key to our approach is to relate the problem of finding a local minimum of a cubic polynomial $p$ to the following optimization problem:
	\begin{equation}\label{Eq: convex pop}
    \begin{aligned}
	& \underset{x \in \Rn}{\inf}
	& & p(x) \\
	& \text{subject to}
	&& \Hx \succeq 0.\\
	\eaeql
	The connection between solutions of (\ref{Eq: convex pop}) and local minima of $p$ is established by Theorem \ref{Thm: Closure} below. The feasible set of (\ref{Eq: convex pop}) has interesting geometric properties (see, e.g., Corollary \ref{Thm: Cubic Spectrahedron}) and will be referred to with the following terminology in the remainder of the paper.
	
	\begin{defn}\label{Def: ConvR}
		The \emph{convexity region} of a polynomial $p: \Rn \to \R$ is the set $$CR_p \defeq \{x \in \Rn\ |\ \Hx \succeq 0\}.$$
	\end{defn}
	Observe that for any cubic polynomial, its convexity region is a spectrahedron,
	and thus a convex set. As $p$ is a convex function when restricted to its convexity region, one can consider (\ref{Eq: convex pop}) to be a convex problem in spirit. 
	
	\begin{theorem}\label{Thm: Closure}
        Let $p$ be a cubic polynomial with a second-order point. Then the following sets are equivalent:
        \renewcommand{\labelenumi}{(\roman{enumi})}
		\begin{enumerate}
			\item $SO_p$
			\item Minima of (\ref{Eq: convex pop}).
		\end{enumerate}
		Furthermore, if $p$ has a local minimum, then these two sets are equivalent to:
		\renewcommand{\labelenumi}{(\roman{enumi})}
		\begin{enumerate}\addtocounter{enumi}{2}
			\item $\overline{LM_p}$.
		\end{enumerate}
	\end{theorem}
	\begin{proof}
	    $(i) \subseteq (ii)$.\\
		Let $\ybar \in SO_p$ and $\xbar$ be any feasible point to (\ref{Eq: convex pop}). If we consider the univariate cubic polynomial $q(\alpha) \defeq p(\xbar + \alpha(\ybar - \bar{x}))$, i.e., the restriction of $p$ to the line passing through $\xbar$ and $\ybar$, we can see that $\alpha = 1$ is a second-order point of $q$. Note that if any univariate cubic polynomial has a second-order point, then that second-order point is a minimum of it over its convexity region. In particular, because $\xbar$ is feasible to (\ref{Eq: convex pop}) and thus $\alpha = 0$ is in the convexity region of $q$, we have $p(\ybar) = q(1) \le q(0) = p(\xbar)$. As $\ybar$ is feasible to (\ref{Eq: convex pop}) and has objective value no higher than any other feasible point, it must be optimal to (\ref{Eq: convex pop}).
	
	    $(ii) \subseteq (i)$\\
		Let $\ybar$ be a minimum of (\ref{Eq: convex pop}) (we know that such a point exists because we have shown $SO_p$ is a subset of the minima of (\ref{Eq: convex pop}), and $SO_p$ is nonempty by assumption). Let $\xbar \in SO_p$ and $d \defeq \ybar - \xbar$. Observe that $p(\ybar) = p(\xbar)$, and so by Lemma \ref{Lem: constant null}, we must have $d \in \nulls(\Hxb)$. It follows that $\grad p(\ybar) = \grad p_3(d)$ (cf. the proof of Lemma \ref{Lem: Gradient Invariance}). Now suppose for the sake of contradiction that $\ybar$ is not a second-order point. Since $\ybar$ is feasible to (\ref{Eq: convex pop}), we must have $\grad p(\ybar) = \grad p_3(d) \ne 0$. As $p(\xbar) = p(\xbar + \alpha d)$ for any scalar $\alpha$ due to Lemma \ref{Lem: constant null}, we must have $p_3(d) = \frac{1}{6} d^T \Hess p_3(d)d = 0$ (see (\ref{Eq: cubic taylor})). Thus we can write
		\begin{align*}
		\big(d - \alpha \gp3d\big)^T\Hess p(\ybar)\big(d-\alpha \gp3d\big) =& \big(d - \alpha \gp3d\big)^T\big(\Hxb + \Hess p_3(d)\big)\big(d-\alpha \gp3d\big)\\
		=& \alpha^2 \gp3d^T\Hxb \gp3d - 2 \alpha \gp3d^T \Hess p_3(d)^Td\\
		&+ \alpha^2 \gp3d^T\Hess p_3(d)\gp3d\\
		=& \alpha^2 \left(\gp3d^T\Hxb \gp3d + \gp3d^T\Hess p_3(d)^T\gp3d\right)\\
		&- 4\alpha \gp3d^T\gp3d,
		\end{align*}
		where the last equality follows from that $\grad p_3(d) = \frac{1}{2} \Hess p_3(d)^Td$ due to Euler's theorem for homogeneous functions. Note that the right-hand side of the above expression is negative for sufficiently small $\alpha > 0$, and so $\Hess p(\ybar)$ is not psd, which contradicts feasibility of $\ybar$ to (\ref{Eq: convex pop}).
		\\
		
		For the second claim of the theorem, suppose that $p$ has a local minimum. The following arguments will show $(i) = (ii) = (iii).$
		
		$(iii) \subseteq (i)$\\
		Clearly any local minimum of $p$ is a second-order point. Since the gradient and the Hessian of $p$ are continuous in $x$ and as the cone of psd matrices is closed, the limit of any convergent sequence of second-order points is a second-order point.
		

		$(ii) \subseteq (iii)$.\\
		Let $\ybar$ be any minimum of (\ref{Eq: convex pop}). Consider any local minimum $\xbar$ of $p$ and let $z_\alpha \defeq \xbar + \alpha(\ybar-\xbar)$. As both $\Hess p(\ybar)$ and $\Hess p(\xbar)$ are psd, any point $z_\alpha$ with $\alpha \in [0,1)$ satisfies the SONC and TOC, by the same arguments as in the proof of Theorem \ref{Thm: WLM Convex}. Now note that since $\xbar$ is a second-order point, it is also a minimum of (\ref{Eq: convex pop}) (as $(i) \subseteq (ii)$) and thus $p(\ybar) = p(\xbar)$. From Lemma \ref{Lem: constant null}, we then have $\ybar - \xbar \in \nulls(\Hxb)$, and so from Lemma \ref{Lem: Gradient Invariance}, $z_\alpha$ satisfies the FONC for any $\alpha$. Thus, in view of Theorem \ref{Thm: TOC}, for any $\alpha \in [0,1), z_\alpha$ is a local minimum of $p$. Therefore $\ybar$ can be written as the limit of a sequence of local minima (i.e., $\{z_\alpha\}$ as $\alpha \to 1$), and hence $\ybar \in \overline{LM_p}$.
	\end{proof}
	
	\begin{remark}\label{Rem: SOP Spectrahedron}
	Note that as a consequence of Theorems \ref{Thm: WLM Closure Spectrahedron}	and \ref{Thm: Closure}, if a cubic polynomial $p$ has a local minimum, then $SO_p$ is a spectrahedron. In fact, $SO_p$ is a spectrahedron for \emph{any} cubic polynomial $p$; see Theorem~\ref{Thm: solution recovery sdp}. In that theorem, we will give a more useful spectrahedral representation of $SO_p$ which does not rely on knowledge of a local minimum.
	\end{remark}

	\begin{cor}\label{Cor: optval convex pop}
	Let $p$ be a cubic polynomial with a second-order point. Then the optimal value of (\ref{Eq: convex pop}) is the value that $p$ takes at any of its second-order points (and in particular, at any of its local minima if they exist).
	\end{cor}
	\begin{proof}
	This is immediate from the equivalence of $(i)$ and $(ii)$ in Theorem \ref{Thm: Closure}.
	\end{proof}

	\subsection{Distinction Between Local Minima and Second-Order Points}

	We have shown that the optimization problem in (\ref{Eq: convex pop}) gives an approach for finding second-order points of a cubic polynomial $p$ without computing its critical points. However, not all second-order points are local minima, and so in this subsection, we characterize the difference between the two notions more precisely. We first recall the concept of the relative interior of a (convex) set (see, e.g., \cite[Chap. 6]{rockafellar1970convex}).
	
	\begin{defn}\label{Def: Relative interior}
		The relative interior of a nonempty convex set $S \subseteq \Rn$ is the set
		$$ri(S) \defeq \{x \in S\ |\ \forall y \in S, \exists \lambda > 1\ s.t.\ \lambda x + (1-\lambda)y \in S \}.$$
	\end{defn}
	This definition generalizes the notion of interior to sets which do not have full dimension. One can show that for a convex set $S$, $ri(S)$ is convex, $ri(\bar{S}) = ri(S)$, and $\overline{ri(S)} = \bar{S}$~\cite{rockafellar1970convex}. In general, for a nonempty convex set $S$, we have $ri(\bar{S}) = ri(S) \subseteq S$, but we may not have $ri(\bar{S}) = S$. (For example, let $S$ be a line segment with one of its endpoints removed.) It turns out, however, that for a cubic polynomial $p$ with a local minimum, $ri(\overline{LM_p}) = LM_p$. 
	
	\begin{theorem}\label{Thm: Local Minima SOP}
		Let $p: \Rn \to \R$ be a cubic polynomial with a local minimum. Then the following three sets are equivalent:
		\renewcommand{\labelenumi}{(\roman{enumi})}
		\begin{enumerate}
			\item $LM_p$
			\item $ri(SO_p)$
			\item Intersection of critical points of $p$ with $ri(CR_p)$.
		\end{enumerate}
	\end{theorem}

	\begin{proof}
		$(ii) \subseteq (i)$\\
		Recall from Theorem \ref{Thm: WLM Convex} that $LM_p$ is convex, and from Theorem \ref{Thm: Closure} that $SO_p = \overline{LM_p}$. Then we have $ri(SO_p) = ri(\overline{LM_p}) = ri(LM_p) \subseteq LM_p$.
		
		$(i) \subseteq (ii)$\\
		We prove the contrapositive. Let $\xbar$ be a point which is not in $ri(SO_p)$. If $\xbar$ is not a second-order point, then it clearly cannot be a local minimum. Suppose now that $\xbar \in SO_p \backslash ri(SO_p)$. Then there is another second-order point $\ybar$ such that $\ybar + \lambda (\xbar - \ybar)$ is not a second-order point for any $\lambda > 1$. Note from Lemma~\ref{Lem: constant null} and the statement after it that $p(\ybar + \lambda (\xbar - \ybar))$ is a constant univariate function of $\lambda$. Now for any $\epsilon > 0$, define the point $\bar{z}_\epsilon \defeq \xbar + \frac{\epsilon}{2 \|\xbar - \ybar\|} (\xbar - \ybar)$. Since $\bar{z}_\epsilon$ is not a second-order point and thus not a local minimum, there is a point $z_\epsilon$ satisfying $\|\bar{z}_\epsilon - z_\epsilon\| < \frac{\epsilon}{2}$ and $$p(z_\epsilon) < p(\bar{z}_\epsilon) = p(\frac{\epsilon}{2 \|\xbar - \ybar\|} (\xbar - \ybar)) = p(\xbar).$$ Furthermore, by the triangle inequality, $z_\epsilon$ also satisfies $\|z_\epsilon - \xbar\| < \epsilon$. Thus, by considering $\{z_\epsilon\}$ as $\epsilon \to 0$, we can conclude that $\xbar$ is not a local minimum.
		
		$(i) \subseteq (iii)$\\
		Consider any local minimum $\xbar$ of $p$, which clearly must also be a critical point of $p$, and a member of $CR_p$. Suppose for the sake of contradiction that $\xbar \not\in ri(CR_p)$. Then there exists $y \in CR_p$ such that for any scalar $\alpha > 0, \Hess p(\xbar + \alpha(\xbar - y))$ is not psd. In particular, for any $\alpha > 0$ there exists a unit vector $z_\alpha \in \Rn$ such that $z_\alpha^T \Hess p(\xbar + \alpha(\xbar - y)) z_\alpha < 0$.
		
		We now show that for any $\alpha$, $z_\alpha$ can be taken to be in $\cols(\Hxb)$. This is because, as we will show, if $z_\alpha = d + v$, where $d \in \nulls(\Hxb)$ and $v \in \cols(\Hxb)$,
		\beq\label{Eq: Cross Term Dies}
		(d+v)^T\Hess p(\xbar + \alpha(\xbar - y))(d+v) = v^T\Hess p(\xbar + \alpha(\xbar - y))v.\eeq
		Observe that if $p$ is written in the form (\ref{Eq: Cubic Poly Form}), for any $d~\in~\nulls(\Hxb)$, we have
		\baeq
		d^T \Hess p(\xbar + \alpha(\xbar - y)) d
		&= d^T\left(\sum_{i=1}^n (\xbar_i + \alpha(\xbar_i - y_i))H_i  + Q\right)d\\
		&= d^T\left(\sum_{i=1}^n \xbar_iH_i + Q\right)d + \alpha \sum_{i=1}^n (d^TH_id)(\xbar_i-y_i) = 0,
		\eaeq
		where the last equality follows from that $d \in \nulls(\Hxb)$, and the TOC, recalling that the $i$-th entry of $\gp3d$ is $\frac{1}{2}d^TH_id$. Note in particular that the expression above also holds for $\alpha = -1$, and so $d \in \nulls(\Hess p(y))$. Now observe that because we can write $$\Hess p(\xbar + \alpha(\xbar - y)) = (1+\alpha)\Hxb - \alpha \Hess p(y),$$
		we have $\Hess p(\xbar + \alpha(\xbar - y))d = 0$. Thus, we have shown (\ref{Eq: Cross Term Dies}), and we can take $z_\alpha \in \cols(\Hxb)$.
		
		Note that if $z_\alpha \in \cols(\Hxb)$, then by Lemma \ref{Lem: Smallest Nonzero Eigenvalue} we have $z_\alpha^T\Hxb z_\alpha \ge \lambda$, where $\lambda$ is the smallest nonzero eigenvalue of $\Hxb$. Thus, for small enough $\alpha$, the quantity $z_\alpha^T \Hess p(\xbar + \alpha(\xbar - y)) z_\alpha$ is positive and so we arrive at a contradiction.
		
		$(iii) \subseteq (i)$\\
		Let $\xbar$ be a critical point which is in $ri(CR_p)$. Clearly $\xbar \in SO_p$. Consider any local minimum $\ybar$ of $p$, and	observe that for any $\alpha \ne 0$, we can write \beq \label{Eq: xbar convex combination} \xbar = \frac{1}{\alpha}(\alpha \xbar + (1-\alpha)\ybar) + \frac{\alpha-1}{\alpha}\ybar.\eeq
		As $\xbar \in ri(CR_p)$ and $\ybar \in CR_p, \alpha \xbar + (1-\alpha) \ybar \in CR_p$ for some $\alpha > 1$. In particular, for that $\alpha, \Hess p(\alpha \xbar + (1-\alpha)\ybar) \succeq 0$ and thus in view of (\ref{Eq: xbar convex combination}), we can see that $\nulls(\Hess p(\xbar)) \subseteq \nulls(\Hess p(\ybar))$. Hence, because the TOC holds at $\ybar$, it must also hold at $\xbar$. Thus $\xbar$ is a local minimum.
	\end{proof}
	
	Figure \ref{Fig: SOP LM} demonstrates the relation between $LM_p$ and $SO_p$ for the polynomial $p(x_1,x_2) = x_1^2x_2$. For this example, $SO_p = \{(x_1,x_2)\ |\ x_1 = 0, x_2 \ge 0\}$, and $LM_p = \{(x_1,x_2)\ |\ x_1 = 0, x_2 > 0\}$ (see Example~\ref{Ex: local min}).
	\begin{figure}[H]
	\centering
	\includegraphics[height=.25\textheight,keepaspectratio]{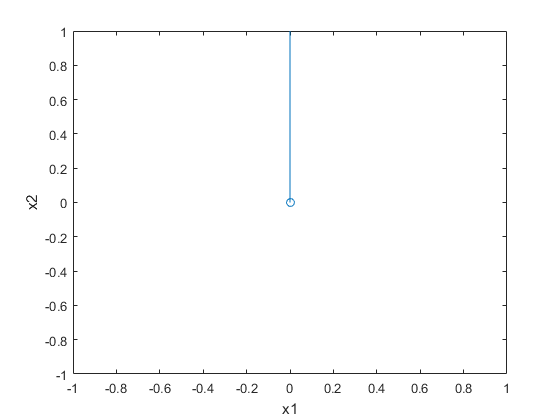}
	\includegraphics[height=.25\textheight,keepaspectratio]{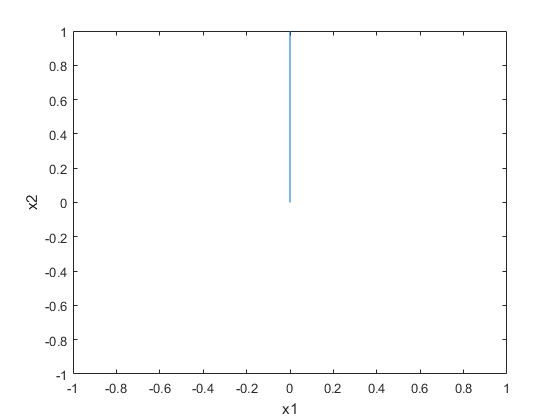}
	\caption{The set of local minima (left) and second-order points (right) of the cubic polynomial $p(x_1,x_2) = x_1^2x_2$. Note that $SO_p$ is the closure of $LM_p$ (Theorem \ref{Thm: Closure}) and $LM_p$ is the relative interior of $SO_p$ (Theorem \ref{Thm: Local Minima SOP}).}
	\label{Fig: SOP LM}
	\end{figure}

	Theorem \ref{Thm: Local Minima SOP} gives rise to the following interesting geometric fact about local minima of cubic polynomials.

	\begin{cor}\label{Cor: Relint NS}
		Let $\xbar$ and $\ybar$ be two local minima of a cubic polynomial. Then $$\nulls (\Hxb) = \nulls (\Hess p(\ybar)).$$
	\end{cor}
	
	\begin{proof}
	It is known (see \cite[Corollary 1]{ramana1995some}) that for a spectrahedron $\{x \in \Rn\ |\ A_0 + \sum_{i=1}^n x_iA_i \succeq 0\}$ and any two points $x$ and $y$ in its relative interior, $\nulls(A_0 + \sum_{i=1}^n x_iA_i) = \nulls(A_0 + \sum_{i=1}^n y_iA_i)$. In view of the facts that for any cubic polynomial $p$, $CR_p$ is a spectrahedron and $LM_p \subseteq ri(CR_p)$ (from Theorem \ref{Thm: Local Minima SOP}), the result is immediate.
	\end{proof}
	
	\subsection{Spectrahedra and Convexity Regions of Cubic Polynomials}\label{SSec: Cubic Spectrahedra}
	
	We end this section with a result relating general spectrahedra and convexity regions of cubic polynomials. Recall from the end of Section \ref{SSec: Cubic Preliminaries} that if $S \defeq \{x \in \Rn\ |\ A_0 + \sum_{i=1}^n x_iA_i \succeq 0\}$ is a special spectrahedron, where $A_0, \ldots, A_n$ are $n \times n$ symmetric matrices satisfying $$(A_i)_{jk} = (A_j)_{ik} = (A_k)_{ij}$$ for any $i,j,k \in \{1 ,\ldots, n\}$, then $S$ is the convexity region of the cubic polynomial $$p(x) = \frac{1}{6}\sum_{i=1}^n x^Tx_iA_ix + \frac{1}{2}x^TA_0x.$$
	
	The following theorem shows that if the number of variables is allowed to increase, then \emph{any} spectrahedron can be represented by the convexity region of a cubic polynomial.
	
	\begin{theorem}\label{Thm: Cubic Spectrahedron}
		Let a spectrahedron $S \subseteq \Rn$ be given by $S \defeq \{x \in \Rn\ |\ A_0 + \sum_{i=1}^n x_iA_i \succeq 0\}$, where $A_0, \ldots, A_n \in \mathbb{S}^{m \times m}$. There exists a cubic polynomial $p$ in at most $m+n$ variables such that $S$ is a projection of its convexity region; i.e., $$S = \{x \in \Rn\ |\ \exists y \in \R^m \mbox{ such that } (x,y) \in CR_p\}.$$ Furthermore, the interior of $S$ is a projection of the set of local minima of $p$.
	\end{theorem}

	\begin{proof}
		Let $A(x) \defeq A_0 + \sum_{i=1}^n x_iA_i$. We first present a characterization of the interior of $S$ following the developments in Section 2.4 of \cite{ramana1995some}. Let $\nulls_A \defeq \nulls(A_0) \cap  \ldots \cap \nulls(A_n)$, and $V$ be a full-rank matrix whose columns span the orthogonal complement of $\nulls_A$. Suppose that $\nulls_A$ is $(m-k)$-dimensional. Then there exist matrices $B_0, \ldots, B_n \in \mathbb{S}^{k \times k}$ with $\nulls(B_0) \cap \ldots \cap \nulls(B_n) = \{0_k\}$ such that $$B(x) \defeq B_0 + \sum_{i=1}^n x_iB_i = V^TA(x)V.$$
		In \cite[Corollary 5]{ramana1995some}, it is shown that $B(x)$
		\beq\label{Eq: Pencils same}\{x \in \Rn\ |\ A(x) \succeq 0\} = \{x \in \Rn\ |\ B(x) \succeq 0\}\eeq
		and that the set $\{x \in \Rn\ |\ B(x) \succ 0\}$ gives the interior of $S$. Now consider the following cubic polynomial in $n+k$ variables:
		\beq\label{Eq: Cubic with LM} p(x,y) \defeq y^TB(x)y.\eeq
		Observe that the partial derivative of $p$ with respect to $y$ is $2B(x)y$, the partial derivative of $p$ with respect to $x_i$ is $y^TB_iy$, and the Hessian of $p$ is
		$$\Hess p(x,y) = 2\bmat 0 & C(y)^T \\ C(y) & B(x) \emat,$$
		where $C(y)$ is an $k \times n$ matrix whose $i$-th column equals $B_iy$. One can then immediately see that if $(\xbar, \ybar) \in CR_p$, then we must have $B(\xbar) \succeq 0$. Conversely, if $B(\xbar) \succeq 0$, then $(\xbar,0_k) \in CR_p$. Hence, in view of (\ref{Eq: Pencils same}), we have shown that the spectrahedron $S$ is the projection of $CR_p$ onto the $x$ variables.
		
		We now show that $LM_p = \{x \in \Rn\ |\ B(x) \succ 0\} \times \{0_k\}$. This would prove the second claim of the theorem. First let $\xbar$ be such that $B(\xbar) \succ 0$. Note that $p(\xbar, 0_k) = 0$ and that for any two vectors $\chi \in \Rn$ and $\psi \in \mathbb{R}^k$, $$p(\xbar + \chi, \psi) = \psi^T\left(B(\xbar) + \sum_{i=1}^n B_i\chi_i\right)\psi.$$ Since $B(\xbar) \succ 0$, then for any $\chi$ of sufficiently small norm, $B(\xbar) + \sum_{i=1}^n B_i\chi_i$ is still positive definite, and hence for any $\psi$, $p(\xbar + \chi,\psi) \ge 0 = p(\xbar,0_k)$. Thus $(\xbar, 0_k)$ is a local minimum of $p$.
		
		Now let $(\xbar, \ybar)$ be a local minimum of $p$. From the SONC, we must have $B(\xbar) \succeq 0$ and $C(\ybar) = 0$, which implies that $B_i\ybar = 0_k, \forall i \in \otn$. Since
		$$\frac{\partial p}{\partial y}(\xbar,\ybar) = 2B(\xbar)\ybar = 2\left(B_0 + \sum_{i=1}^n \xbar_iB_i\right)\ybar = 2B_0\ybar + 2\sum_{i=1}^n \xbar_i(B_i\ybar),$$ it further follows from the FONC that $B_0\ybar = 0$. As $\nulls(B_0) \cap \ldots \cap \nulls(B_n) = \{0_k\}$ by construction, it follows that we must have $\ybar = 0_k$. Next, observe that $\nulls(\Hess p(\xbar, 0_k)) = \R^n \times \nulls(B(\xbar))$. Let $d \in \nulls(B(\xbar))$, and note that for any $i \in \otn$, $(e_i,d) \in \nulls(\Hess p(\xbar, 0_k))$ and $\frac{\partial p_3}{\partial y}(e_i,d) = B_id$. Then from the TOC, we must have $B_id = 0_k, \forall i \in \{1, \ldots, n\}$. Furthermore, since $d \in \nulls(B(\xbar))$, it follows that $B_0d = 0_k$ as well. Again, as $\nulls(B_0) \cap \ldots \cap \nulls(B_n) = \{0_k\}$ by construction, it follows that we must have $d = 0_k$ and thus $B(\xbar) \succ 0$.
	\end{proof}
		
\section{Complexity Justifications for an Exact SDP Oracle}\label{Sec: Complexity}

In the next section, we show that second-order points and local minima of cubic polynomials can be found by solving polynomially-many semidefinite programs with a polynomial number of variables and constraints. One caveat however is that the inputs and outputs of these semidefinite programs can sometimes be algebraic but not necessarily rational numbers. As a result, we cannot claim that second-order points and local minima of cubic polynomials can be found in polynomial time in the Turing model of computation. In this subsection, we give evidence as to why establishing the complexity of these problems in the Turing model is at the moment likely out of reach.

\begin{defn}
The \emph{SDP Feasibility Problem} (SDPF) is the following decision question: Given $m \times m$ symmetric matrices $A_0, \ldots, A_n$ with rational entries, decide whether there exists a vector $x \in \Rn$ such that $A_0 + \sum_{i=1}^n x_iA_i \succeq 0$.
\end{defn}

\begin{defn}
The \emph{SDP Strict Feasibility Problem} (SDPSF) is the following decision question: Given $m \times m$ symmetric matrices $A_0, \ldots, A_n$ with rational entries, decide whether there exists a vector $x \in \Rn$ such that $A_0 + \sum_{i=1}^n x_iA_i \succ 0$.
\end{defn}

Even though semidefinite programs can be solved to arbitrary accuracy in polynomial time~\cite{vandenberghe1996semidefinite}, the complexities of the decision problems above remain as two of the outstanding open problems in semidefinite programming. At the moment, it is not known if these two decision problems even belong to the class NP~\cite{ramana1993algorithmic, porkolab1997complexity, de2006aspects}. We show next that the complexities of these problems are a lower bound on the complexities of testing existence of second-order points and local minima of cubic polynomials. (In Section~\ref{Sec: Finding Local Min}, we accomplish the more involved task of giving the reduction in the opposite direction.)

\begin{theorem}\label{Thm: SOP SDPF}
    If the problem of deciding whether a cubic polynomial has any second-order points is in P (resp. NP), then SDPF is in P (resp. NP).
\end{theorem}

\begin{proof}
    Given matrices $A_0, \ldots, A_n \in \mathbb{S}^{m \times m}$, let $A(x) \defeq A_0 + \sum_{i=1}^n x_iA_i$. By noting that the cubic polynomial $p(x,y) = y^TA(x)y$ has as its Hessian
    $$\Hess p(x,y) = 2\bmat 0 & B(y)^T \\ B(y) & A(x) \emat,$$
	where $B(y)$ is an $m \times n$ matrix whose $i$-th column equals $A_iy$, we can see that if $A(\xbar) \succeq 0$ for some $\xbar \in \Rn$, then $\Hess p(\xbar, 0_k) \succeq 0$. Since $p$ is quadratic in the variables $y$, $\grad p(\xbar, 0_k) = 0_{m+n}$, and hence $(\xbar, 0_k)$ is a second-order point of $p$. Conversely, if $A(x) \not\succeq 0$ for any $x \in \Rn$, then clearly $\Hess p(x, y) \not\succeq 0$ for any $x \in \Rn$ and $y \in \R^m$, and thus $p$ cannot have any second-order points.
	
	The above reduction shows that any polynomial-time algorithm (or polynomial-time verifiable certificate) for existence of second-order points of cubic polynomials translates into one for SDPF.
\end{proof}

\begin{theorem}\label{Thm: LM SDPSF}
	If the problem of deciding whether a cubic polynomial has any local minima is in P (resp. NP), then SDPSF is in P (resp. NP).
\end{theorem}

\begin{proof}
	Given matrices $A_0, \ldots, A_n \in \mathbb{S}^{m \times m}$, let $A(x) \defeq A_0 + \sum_{i=1}^n x_iA_i$ and consider the set $S \defeq \{x \in \Rn\ |\ A(x) \succeq 0\}$. It is not difficult to see that there exists $\xbar \in \Rn$ such that $A(\xbar) \succ 0$ if and only if $S$ has a nonempty interior and $\nulls_A \defeq \nulls(A_0) \cap \nulls(A_1) \cap \nulls(A_2) \ldots \cap \nulls(A_n) = \{0_m\}$.\footnote{The ``only if'' direction is straightforward and the ``if'' direction follows from \cite[Corollary 5]{ramana1995some}.} The latter condition can be checked in polynomial time by solving linear systems. The former can be reduced---due to the second claim of Theorem \ref{Thm: Cubic Spectrahedron}---to deciding if the cubic polynomial constructed in (\ref{Eq: Cubic with LM}) has a local minimum. Note that the polynomial in (\ref{Eq: Cubic with LM}) has coefficients polynomially sized in the entries of the matrices $A_i$, since the matrix $V$ in the proof of Theorem \ref{Thm: Cubic Spectrahedron} can be taken to be the identity matrix when $\nulls_A = \{0_m\}$.
\end{proof}

In addition to the difficulties alluded to in the above two theorems, the following three examples point to concrete representation issues that one encounters in the Turing model when dealing with local minima of cubic polynomials. The same complications are known to arise for SDP feasibility problems~\cite{de2006aspects}.

\begin{example}
	{\bf A cubic polynomial with only irrational local minima.} Consider the univariate cubic polynomial $p(x) = x^3 - 6x$. One can easily verify that its unique local minimum is at $x=\sqrt{2},$ which is irrational even though the coefficients of $p$ are rational.
\end{example}

\begin{example}
{\bf A cubic polynomial with an irrational convexity region.} Consider the quintary cubic polynomial $p(x,y) = y^TA(x)y$, where
$$A(x) = \bmat 2 & x & 0 & 0 \\
x & 1 & 0 & 0 \\
0 & 0 & 2x & 2 \\
0 & 0 & 2 & x\emat.$$
One can easily verify that $x = \sqrt{2}$ is the only scalar satisfying $A(x) \succeq 0$. Since the matrix $2A(x)$ is a principal submatrix of $\Hess p(x,y)$, any point in the convexity region of $p$ must satisfy $x = \sqrt{2}$ (even though the coefficients of $p$ are rational).
\end{example}

\begin{example}
	{\bf A family of cubic polynomials whose local minima have exponential bitsize.} Consider the family of cubic polynomials $p_n(x,y) = y^TA_n(x)y$ in $3n$ variables, where
	$$A_n(x) = \bmat x_1 & 2 & 0 & 0 & \cdots & 0 & 0\\
	2 & 1 & 0 & 0 & \cdots & 0 & 0 \\
	0 & 0 & x_2 & x_1 & \cdots & 0 & 0\\
	0 & 0 & x_1 & 1 & \cdots & 0 & 0 \\
	\cdots & \cdots & \cdots & \cdots & \ddots & \cdots & \cdots \\
	0 & 0 & 0 & 0 & \cdots & x_n & x_{n-1}\\
	0 & 0 & 0 & 0 & \cdots & x_{n-1} & 1\emat.$$
	We show that even though these polynomials have some rational local minima, it takes exponential time to write them down. From the proof of Theorem \ref{Thm: Cubic Spectrahedron}, one can infer that the set of local minima of $p_n$ is the set $\{x \in \Rn\ |\ A_n(x) \succ 0\} \times \{0_{2n}\}$. However, observe that to have $A_n(x) \succ 0$ (or even $A_n(x) \succeq 0)$, we must have $$x_1 \ge 4, x_2 \ge 16, \ldots, x_n \ge 2^{2^n}.$$  Hence, any local minimum of $p_n$ has bit length at least $O(2^n)$ even though the bit length of the coefficients of $p_n$ is $O(n)$.
\end{example}

\section{Finding Local Minima of Cubic Polynomials}\label{Sec: Finding Local Min}

In this section, we derive an SDP-based approach for finding second-order points and local minima of cubic polynomials. This, along with the results established in Section \ref{Sec: NP-hardness results}, will complete the entries of Table \ref{Table: Complexity Existence} from Section~\ref{Sec: Introduction}. We begin with some preliminaries that are needed to present the theorems of this section.

\subsection{Preliminaries from Semidefinite and Sum of Squares Optimization}\label{SSec: SDP Prelims}

\subsubsection{The Oracle E-SDP}\label{SSSec: ESDP}

Recall that a \emph{spectrahedron} is a set of the type
$$\left\{x \in \Rn\ |\ A_0 + \sum_{i=1}^n x_iA_i \succeq 0\right\},$$
where $A_0, \ldots, A_n$ are symmetric matrices of some size $m \times m$. A \emph{semidefinite representable set} (also known as a \emph{spectrahedral shadow}) is a set of the type
\beq\label{Eq: SDR Set}\left\{x \in \Rn\ |\ \exists y \in \R^k \text{ such that } A_0 + \sum_{i=1}^n x_iA_i + \sum_{i=1}^k y_iB_i \succeq 0\right\},\eeq
for some integer $k \ge 0$ and symmetric $m \times m$ matrices $A_0, A_1, \ldots, A_n, B_1, \ldots, B_k$. These are exactly sets which semidefinite programming can optimize over.

We show in Theorem \ref{Thm: solution recovery sdp} and Corollary \ref{Cor: Complete Cubic SDP SOP} that the set of second-order points of any cubic polynomial is a spectrahedron and describe how a description of this spectrahedron can be obtained from the coefficients of $p$ only.\footnote{Recall that the results of Section \ref{Sec: Geometry} by contrast established spectrahedrality of the set of second-order points under the assumption of existence of a local minimum (see Remark \ref{Rem: SOP Spectrahedron}). Furthermore, the spectrahedral representation that we gave there (see Theorem \ref{Thm: WLM Closure Spectrahedron}) required knowledge of a local minimum.} Since relative interiors of semidefinite representable sets (and in particular spectrahedra) are semidefinite representable \cite[Theorem 3.8]{netzer2010semidefinite}, it follows from our Theorem \ref{Thm: Local Minima SOP} that the set of local minima of any cubic polynomial is semidefinite representable.

Due to the complexity results and representation issues presented in Section \ref{Sec: Complexity}, we assume in this section that we can do arithmetic over real numbers and have access to an oracle which solves SDPs exactly. This oracle---which we call \emph{E-SDP}---takes as input an SDP with real data and outputs the optimal value as a real number if it is finite, or reports that the SDP is infeasible, or that it is unbounded.\footnote{Though this will not be needed for our purposes, it is straightforward to show that for an SDP with $n$ scalar variables, the oracle E-SDP can be called twice to test attainment of the optimal value, and a total of $n+1$ times to recover an optimal solution.} The following lemma shows that E-SDP can find a point in the relative interior of a semidefinite representable set. This will be relevant for us later in this section when we search for local minima of cubic polynomials.

\begin{lem}\label{Lem: Relint recovery}
	Let $S$ be a nonempty semidefinite representable set in $\R^n$. Then a point in $ri(S)$ can be recovered in $2n$ calls to E-SDP.
\end{lem}

\begin{proof}
    Consider the following procedure. Let $S_1 = S$, and for $i \in \otn$ let $$S_{i+1} = S_i \cap \{x \in \Rn\ |\ x_i = x_i^*\},$$ where the scalar $x_i^*$ is chosen to be any ``intermediate'' value of $x_i$ on $S_i$. More precisely, let $\xbar_i$ (resp. $\underline{x}_i$) be the supremum (resp. infimum) of $x_i$ over $S_i$ (these two values may or may not be finite). If $\bar{x}_i = \underline{x}_i$, then set $x_i^* = \bar{x}_i$. Otherwise, set $x_i^*$ to be any scalar satisfying $\underline{x}_i < x_i^* < \bar{x}_i$. Note that for each $i$, $x_i^*$ can be computed using $2$ calls to E-SDP. Hence, after $2n$ calls to E-SDP, we arrive at a set $S_{n+1}$ which is a singleton by construction. 
		
	We next show, by induction, that the point in $S_{n+1}$ belongs to $ri(S)$. First note that as $S$ is nonempty, $ri(S)$ is nonempty~\cite[Theorem 6.2]{rockafellar1970convex}, which implies that $S_1 \cap ri(S) = ri(S)$ is nonempty. Now suppose that $S_i \cap ri(S)$ is nonempty for $i \in \{1,\ldots, k\}$. We show that $S_{k+1} \cap ri(S)$ is nonempty.
	
	First suppose that $k$ is such that $\bar{x}_k = x_k^* = \underline{x}_k$. In this case, because $\forall x \in S_k, x_k = x_k^*$,
	$$S_{k+1} \cap ri(S) = S_k \cap \{x \in \Rn\ |\ x_k = x_k^*\} \cap ri(S) = S_k \cap ri(S) \ne \emptyset.$$
	Now suppose that $\underline{x}_k < x_k^* < \bar{x}_k$. 
	By the definition of $\xbar_k$, there exists a sequence of points $\{y_j\} \subseteq S_k$ such that $(y_j)_k \to \bar{x}_k$. We recall that for any $z \in ri(S), y \in \bar{S}$, and $\lambda \in (0,1]$, $\lambda z + (1-\lambda)y \in ri(S)$~\cite[Theorem 6.1]{rockafellar1970convex}. Now let $z \in S_k \cap ri(S)$. Since $S_k$ is convex, for any $y \in S_k \cap \bar{S}$ and $\lambda \in (0,1]$, $\lambda z + (1-\lambda)y \in S_k \cap ri(S)$. In particular, since $S_k \cap \bar{S} = S_k$, the sequence $\{z_j\} \defeq \{\frac{1}{j}z + \frac{j-1}{j}y_j\}$ satisfies $\{z_j\} \subseteq S_k \cap ri(S)$ and $(z_j)_k \to \xbar_k$. Similarly, there exists a sequence of points $\{w_j\} \subseteq S_k \cap ri(S)$ such that $(w_j)_k \to \underline{x}_k$. As $S_k \cap ri(S)$ is convex, there must then be a point $x \in S_k \cap ri(S)$ satisfying $x_k = x_k^*$, and so $$S_{k+1} \cap ri(S) = S_k \cap \{x \in \Rn\ |\ x_k = x_k^*\} \cap ri(S)$$ is not empty.
\end{proof}

\subsubsection{Overview of Sum of Squares Polynomials}\label{SSSec: Sos}
In order to describe our SDP-based approach for finding local minima of cubic polynomials, we also need to briefly review the connection between sum of squares polynomials and matrices to semidefinite programming. We remark that in related work \cite{nie2015hierarchy}, the author produces a hierarchy of SDPs of growing size, based also on the connection with sum of squares polynomials, which allows him to find local minima of polynomials of any degree in the limit of his hierarchy. However, no claims are established on the level of the hierarchy needed to recover a local minimum (except for finiteness under some assumptions). Our contribution is to derive a new SDP relaxation for the case of cubic polynomials, which has small size, and is guaranteed to find a local minimum.


We say that a (multivariate) polynomial $p:\R^n\to\R$ is \emph{nonnegative} if $p(x) \ge 0, \forall x \in \R^n$. A polynomial $p$ is said to be a \emph{sum of squares} (sos) if $p= \sum_{i=1}^r q_i^2$ for some polynomials $q_1,\ldots,q_r$. This is an algebraic sufficient, but in general not necessary~\cite{Hilbert_1888}, condition for nonnegativity of a polynomial. While deciding nonnegativity of a polynomial is in general NP-hard (see, e.g.,~\cite{murty1987some}), one can decide whether a polynomial is sos via semidefinite programming. This is because a polynomial $p$ of degree $2d$ in $n$ variables is a sum of squares if and only if there exists an ${n+d \choose d} \times {n+d \choose d}$ positive semidefinite matrix $Q$ satisfying the identity
\beq \label{Eq: Sos Gram} p(x) = z(x)^TQz(x),\eeq
where $z(x)$ denotes the vector of all monomials in $x$ of degree less than or equal to $d$. Note that because of this equivalence, one can also require a polynomial $p$ with unknown coefficients to be sos in a semidefinite program. Given a rank-$r$ psd matrix $Q$ that satisfies (\ref{Eq: Sos Gram}), one can write $Q$ as $\sum_{i=1}^r v_iv_i^T$ (e.g. via a Cholesky or an eigenvalue factorization), and obtain an sos decomposition of $p$ as $p=\sum_{i=1}^r (v_i^Tz(x))^2$.

The notion of sum of squares also extends to polynomial matrices (i.e., matrices whose entries are multivariate polynomials). We say that symmetric polynomial matrix $M(x):\Rn \to \R^m \times \R^m$ is an \emph{sos-matrix} if it has a factorization as $M(x)=R(x)^TR(x)$ for some $r\times m$ polynomial matrix $R(x)$~\cite{helton2010semidefinite}. Observe that if $M$ is an sos-matrix, then $M(x) \succeq 0$ for any $x \in \Rn$. One can check that $M(x)$ is an sos-matrix if and only if the scalar-valued polynomial $y^TM(x)y$ in variables $(x_1,\ldots,x_n,y_1,\ldots,y_m)$ is sos. Indeed, the ``only if'' direction is clear, the ``if'' direction is because when $y^TM(x)y = \sum_{i=1}^r q_i^2(x,y)$ for some polynomials $q_1, \ldots, q_r$, each $q_i$ must be linear in $y$ and thus writable as $q_i(x)=\sum_{j=1}^m y_jq_{ij}(x)$ for some polynomials $q_{ij}$. Then if $R(x)$ is the $r \times m$ matrix where $R_{ij}(x) = q_{ij}(x)$, we will have $M(x) = R^T(x)R(x)$.

\subsection{A Sum of Squares Approach for Finding Second-Order Points}\label{SSec: SDP Approach}

We have shown in Theorem \ref{Thm: Closure} that if a cubic polynomial $p$ has a second-order point, the solutions of the optimization problem in (\ref{Eq: convex pop}) exactly form the set $SO_p$ of its second-order points. The same theorem further showed that if $p$ has a local minimum, then the solutions of (\ref{Eq: convex pop}) also coincide with $\overline{LM_p}$, i.e. the closure of the set of its local minima. Our goal in this section is to develop a semidefinite representation of $SO_p$ which can be obtained directly from the coefficients of $p$ (Corollary \ref{Cor: Complete Cubic SDP SOP}). To arrive to this representation, we first present an sos relaxation of problem (\ref{Eq: convex pop}), which we prove to be tight when $SO_p$ is nonempty (Theorem \ref{Thm: cubic sdp}). We then provide a more efficient representation of the SDP underlying this sos relaxation in Section \ref{SSec: Simplification}. This will lead to an algorithm (Algorithm~\ref{Alg: Complete Cubic SDP}) for finding local minima of cubic polynomials which is presented in Section \ref{SSSec: SOP and Algorithm}.

\begin{theorem}\label{Thm: cubic sdp}
	If a cubic polynomial $p: \Rn \to \R$ has a second-order point, the optimal value of the following semidefinite program\footnote{To clarify, $x$ is not a decision variable in this problem. The decision variables are $\gamma$, the coefficients of $\sigma$, and the coefficients of the entries of $S$. The identity in the first constraint must hold for all $x$, and this can be enforced by matching the coefficient of each monomial on the left with the corresponding coefficient on the right.} is attained and is equal to the value of $p$ at all second-order points:
	
	\begin{equation}\label{Eq: cubic sdp}
    \begin{aligned}
	& \underset{\gamma \in \R, \sigma(x), S(x)}{\sup}
	& & \gamma \\
	& \text{\emph{subject to}}
	&& p(x) - \gamma = \sigma(x) + \Tr(S(x)\Hx),\\
	&&& \sigma(x) \text{\emph{ is a degree-2 sos polynomial}},\\
	&&& S(x) \text{\emph{ is an }} n \times n \text{\emph{ sos-matrix with degree-2 entries.}}\\
	\eaeql
	
\end{theorem}

\begin{proof}
Let $\xbar$ be a second-order point of $p$ and $\gamma^*$ be the optimal value of (\ref{Eq: cubic sdp}). Consider any feasible solution $(\gamma, \sigma, S)$ to (\ref{Eq: cubic sdp}) (nonemptiness of the feasible set is established in the next paragraph). Since $\Hxb \succeq 0$ and $S(\xbar) \succeq 0$, we have $\Tr(\Hxb S(\xbar)) \ge 0$. Since $\sigma(\xbar) \ge 0$ as well, it follows that $p(\xbar) \ge \gamma$. Hence, $p(\xbar) \ge \gamma^*$.

To show that $p(\xbar) \le \gamma^*$ and that the value $\gamma^* = p(\xbar)$ is attained, we establish that
$$(\gamma,\sigma,S) = \left(p(\xbar), \frac{1}{3}(x-\xbar)^T\Hxb(x-\xbar), \frac{1}{6}(x-\xbar)(x-\xbar)^T\right)$$
is feasible to (\ref{Eq: cubic sdp}). Note that $\frac{1}{3}(x-\xbar)^T\Hxb(x-\xbar)$ is an sos polynomial (as $\Hxb$ can be factored into $V^TV$), and that $\frac{1}{6}(x-\xbar)(x-\xbar)^T$ is an sos-matrix by construction. To show that the first constraint in (\ref{Eq: cubic sdp}) is satisfied, consider the Taylor expansion of $p$ around $\xbar$ in the direction $x - \xbar$ (see (\ref{Eq: cubic taylor}), noting that $\gxb = 0$):
\beq \label{Eq: SOP cubic taylor}p(\xbar + (x-\xbar)) = p(\xbar) + \frac{1}{2}(x-\xbar)^T\Hxb(x-\xbar) + p_3(x-\xbar).\eeq
Observe that if $p$ is written in the form (\ref{Eq: Cubic Poly Form}), then we have
\begin{align*}
p_3(x-\xbar) &= \frac{1}{6}(x-\xbar)^T\left(\sum_{i=1}^n (x_i-\xbar_i)H_i\right)(x-\xbar)\\
&=\frac{1}{6}(x-\xbar)^T\left(\sum_{i=1}^n (x_i-\xbar_i)H_i + Q - Q\right)(x-\xbar)\\
&=\frac{1}{6}(x-\xbar)^T\left(\sum_{i=1}^n x_iH_i+ Q - \sum_{i=1}^n \xbar_iH_i - Q\right)(x-\xbar)\\
&=\frac{1}{6}(x-\xbar)^T\Hx(x-\xbar) - \frac{1}{6}(x-\xbar)^T\Hxb(x-\xbar).
\end{align*}
Note further that due to the cyclic property of the trace, we have
$$\frac{1}{6}(x-\xbar)^T\Hx(x-\xbar) = \Tr\left((\frac{1}{6}(x-\xbar)(x-\xbar)^T)\Hx\right).$$
Hence, (\ref{Eq: SOP cubic taylor}) reduces to the following identity
\beq\label{Eq: cubic sos identity}
	p(x) - p(\xbar) = \frac{1}{3}(x-\xbar)^T\Hxb(x-\xbar) + \Tr\left((\frac{1}{6}(x-\xbar)(x-\xbar)^T)\Hx\right),
\eeq
and thus the claim is established.
\end{proof}

Since (\ref{Eq: cubic sdp}) is a tight sos relaxation of (\ref{Eq: convex pop}) when $SO_p$ is nonempty, it is interesting to see how an optimal solution to (\ref{Eq: convex pop}) can be recovered from an optimal solution to (\ref{Eq: cubic sdp}). This is shown in the next theorem, keeping in mind that optimal solutions to (\ref{Eq: convex pop}) are second-order points of $p$ (see Theorem \ref{Thm: Closure}).

\begin{theorem}\label{Thm: solution recovery sdp}
		Let $p: \Rn \to \R$ be a cubic polynomial with a second-order point, and let $(\gamma^*, \sigma^*, S^*)$ be an optimal \edit{solution\footnote{\edit{By Theorem~\ref{Thm: cubic sdp}, for any cubic polynomial with a second-order point, an optimal solution to (\ref{Eq: cubic sdp}) exists.}}} of (\ref{Eq: cubic sdp}) applied to $p$. Then, the set
		\beq\label{Eq: second order points}
		\Gamma \defeq \{x \in \Rn\ |\ \Hx \succeq 0, \sigma^*(x) = 0, \Tr(S^*(x)\Hx) = 0\}
		\eeq
		is a spectrahedron, and $\Gamma = SO_p$.
\end{theorem}

\begin{proof}

We first show that $\Gamma = SO_p$. Let $\xbar$ be a second-order point of $p$. From Theorem \ref{Thm: cubic sdp} and the first constraint of (\ref{Eq: cubic sdp}) we have
$$0 = p(\xbar) - p(\xbar) = p(\xbar) - \gamma^* = \sigma^*(\xbar) + \Tr(S^*(\xbar)\Hxb).$$
As $\sigma^*(\xbar)$ and $\Tr(S^*(\xbar)\Hxb)$ are both nonnegative, the above equation implies they must both be zero, and hence $SO_p \subseteq \Gamma$. To see why $\Gamma \subseteq SO_p$, let $\ybar$ be a point in $\Gamma$ and $\hat{x}$ be an arbitrary second-order point (which by the assumption of the theorem exists). Observe from Theorem \ref{Thm: cubic sdp} and the first constraint of (\ref{Eq: cubic sdp}) that
$$p(\ybar) - p(\hat{x}) = p(\ybar) - \gamma^* = \sigma^*(\ybar) + \Tr(S^*(\ybar)\Hess p(\ybar)) = 0.$$ Additionally, because $\Hess p(\ybar) \succeq 0$, it follows from Corollary \ref{Cor: optval convex pop} that $\ybar$ is optimal to (\ref{Eq: convex pop}), and thus is a second-order point by Theorem \ref{Thm: Closure}.
	
Now we show that $\Gamma$ is a spectrahedron by ``linearizing'' the quadratic and cubic equations that appear in (\ref{Eq: second order points}). Since $\sigma^*$ is a quadratic sos polynomial, it can be written as $\sigma^*(x) = \sum_{i=1}^m q_i^2(x)$ for some affine polynomials $q_1,\ldots,q_m$. Similarly, since $S^*$ is an sos-matrix with quadratic entries, it can be written as $S^*(x) = R(x)^TR(x)$ for some $k \times n$ matrix $R$ with affine entries. First note that as $\Hx$ is affine in $x$ and $\sigma^*$ is a sum of squares of affine polynomials, the set
$$\{x \in \Rn\ |\ \Hx \succeq 0, \sigma^*(x) = 0\} = \{x \in \Rn\ |\ \Hx \succeq 0, q_1(x) = 0,\ldots, q_m(x) = 0\}$$
is clearly a spectrahedron.

Now let $y$ be any point in $ri(CR_p)$. Such a point exists because $CR_p$ is nonempty by assumption, and relative interiors of nonempty convex sets are nonempty \cite[Theorem 6.2]{rockafellar1970convex}. Now let $r_i$ be the $i$-th column of the matrix $R^T$. We claim that $\Gamma$ is equivalent to the following set:
\beq \label{Eq: SOP SDR}\big\{x \in \Rn\ |\ \Hx \succeq 0, q_1(x) = 0, \ldots, q_m(x) = 0, \Hess p(y) r_1(x) = 0, \ldots, \Hess p(y) r_k(x) = 0\big\}.\eeq
Note that this set is a spectrahedron, and that the final $k$ equality constraints are enforcing that each column of $R^T$ be in the null space of $\Hess p(y)$.

To prove the claim, first let $x$ be in (\ref{Eq: SOP SDR}). Note that $\nulls(\Hess p(y)) \subseteq \nulls(\Hx)$, as $y \in ri(CR_p)$ and so $\Hess p(y) = \lambda \Hx + (1-\lambda) \Hess p(z)$ for some $z \in CR_p$ and $\lambda \in (0,1)$. Then,
$$\Tr(S^*(x)\Hx) = \sum_{i=1}^k r_i^T(x)\Hx r_i(x) = 0.$$
Hence $(\ref{Eq: SOP SDR}) \subseteq (\ref{Eq: second order points})$.

To show the reverse inclusion, let $x$ be a point in (\ref{Eq: second order points}). It is easy to check that $\Tr(AB) = 0$ for two psd matrices $A = C^TC$ and $B$ if and only if the columns of $C^T$ belong to the null space of $B$. Hence, we must have $r_i(x) \in \nulls (\Hx)$. Assume first that $x\in ri(CR_p)$. Then we must have $r_i(x) \in \nulls (\Hx) = \nulls (\Hess p(y))$ as $CR_p$ is a spectrahedron and any two matrices in the relative interior of a spectrahedron have the same null space \cite[Corollary 1]{ramana1995some}. To see why we must also have $r_i(x) \in \nulls (\Hess p(y))$ for any $x \in CR_p\backslash ri(CR_p)$, observe that $\nulls(\Hess p(y))$ is closed, the vector-valued functions $r_i$ are continuous in $x$, and the preimage of a closed set under a continuous function is closed.
\end{proof}

\subsection{A Simplified Semidefinite Representation of Second-Order Points and an Algorithm for Finding Local Minima}\label{SSec: Simplification}

In this subsection, we derive a semidefinite representation of the set $SO_p$, which will be given in (\ref{Eq: Complete Cubic SDP SOP}). In contrast to the semidefinite representation in (\ref{Eq: SOP SDR}), which requires first solving (\ref{Eq: cubic sdp}) and then performing some matrix factorizations, the representation in (\ref{Eq: Complete Cubic SDP SOP}) can be immediately obtained from the coefficients of $p$. To find a second-order point of an $n$-variate cubic polynomial via the representation in (\ref{Eq: Complete Cubic SDP SOP}), one needs to solve an SDP with $\frac{(n+2)(n+1)}{2}$ scalar variables and two semidefinite constraints of size $(n+1) \times (n+1)$. This is in contrast to finding a second-order point via the representation in (\ref{Eq: SOP SDR}), which requires solving two SDPs: (\ref{Eq: cubic sdp}) which has $\left(\frac{n(n+1)}{2}+1\right)\left(\frac{(n+2)(n+1)}{2}\right)+1$ scalar variables and two semidefinite constraints of sizes $(n+1) \times (n+1)$ and $n(n+1) \times n(n+1)$ (coming from the two sos constraints), and then the SDP associated with (\ref{Eq: SOP SDR}), which has $n$ scalar variables and a semidefinite constraint of size $n \times n$. Another purpose of this subsection is to present our final result, which is an algorithm for testing for existence of a local minimum (Algorithm \ref{Alg: Complete Cubic SDP} in Section \ref{SSSec: SOP and Algorithm}).

\subsubsection{A Simplified Sos Relaxation}

Recall from the proof of Theorem \ref{Thm: cubic sdp} that if $p$ has a second-order point $\xbar$, then there is an optimal solution to (\ref{Eq: cubic sdp}) of the form
\beq\label{Eq: Ideal solution}
(\gamma, \sigma, S) = \left(p(\xbar), \frac{1}{3}(x-\xbar)^T\Hxb(x-\xbar), \frac{1}{6}(x-\xbar)(x-\xbar)^T\right).\eeq In particular, for this solution, the coefficients of $\sigma$ and $S$ can both be written entirely in terms of the entries of $\xbar$ and the coefficients of $p$. In what follows, we attempt to optimize over solutions to (\ref{Eq: cubic sdp}) which are of the form in (\ref{Eq: Ideal solution}). However, imposing this particular structure on the solution requires nonlinear equality constraints (in fact, it turns out quadratic constraints suffice). Instead, we will impose an SDP relaxation of these nonlinear constraints and show that the relaxation is exact. We follow a standard technique in deriving SDP relaxations for quadratic programs, where the outer product $xx^T$ of some variable $x$ is replaced by a new matrix variable $X$ satisfying $X - xx^T \succeq 0$. The latter matrix inequality that can be imposed as a semidefinite constraint via the Schur complement~\cite{boyd2004convex}. The variable $\xbar$ will be represented by a variable $y \in \Rn$, and the symmetric matrix variable $Y \in \Snn$ will represent $yy^T$. In addition, we will need another scalar variable $z$.

Assume $p$ is given in the form (\ref{Eq: Cubic Poly Form}), and let us expand $\sigma$ in (\ref{Eq: Ideal solution}) (disregarding the factor $\frac{1}{3}$) as follows:

\begin{align*}
(x-\xbar)^T \Hxb (x - \xbar) &= x^T \left(\sum_{i=1}^n \xbar_iH_i + Q\right)x - 2\xbar^T\left(\sum_{i=1}^n \xbar_iH_i + Q\right)x + \xbar^T\left(\sum_{i=1}^n \xbar_iH_i + Q\right)\xbar\\&= x^T\left(\sum_{i=1}^n \xbar_iH_i + Q\right)x - 2\sum_{i=1}^n \Tr(H_i \xbar \xbar^T)x_i - 2\xbar^TQx + \xbar^T\left(\sum_{i=1}^n \xbar_iH_i + Q\right)\xbar,
\end{align*}
where in the last equality we used Lemma \ref{Lem: Hessian switch}. If we replace any occurrence of $\xbar$ with $y$, any occurrence of $\xbar \xbar^T$ with $Y$ and any occurrence of $\xbar^T(\sum_{i=1}^n \xbar_iH_i + Q)\xbar$ with $z$, we can rewrite the above expression as

\beq \label{Eq: new cubic sdp first term}
\sigma_{Y,y,z}(x) \defeq \sum_{j=1}^n \sum_{k=1}^n \left(\sum_{i=1}^n (H_i)_{jk}y_i + Q_{jk}\right)x_jx_k - 2\sum_{i=1}^n (\Tr(H_iY) + e_i^TQy)x_i + z. \eeq
Similarly, the matrix $S$ in (\ref{Eq: Ideal solution}) can be written as $xx^T - xy^T - yx^T + Y$ (disregarding the factor $\frac{1}{6}$). Note that if $Y - yy^T\succeq 0$, then the matrix $xx^T - xy^T - yx^T + Y \edit{= (x-y)(x-y)^T + (Y-yy^T)}$ is an sos-matrix (as a polynomial matrix in $x$). By making these replacements, we arrive at an SDP which attempts to look for a solution to the sos program in (\ref{Eq: cubic sdp}) which is of the structure in (\ref{Eq: Ideal solution}). This is the following SDP\footnote{Note that $x$ is not a decision variable in this SDP as the first constraint needs to hold for all $x$.}:

\begin{equation}\label{Eq: new cubic sdp}
    \begin{aligned}
	& \underset{\gamma \in \R, Y \in \Snn, y \in \Rn, z \in \R}{\sup}
	& & \gamma \\
	& \text{subject to}
	&& p(x) - \gamma = \frac{1}{3}\sigma_{Y,y,z}(x) + \frac{1}{6}\Tr\left(\Hx(xx^T - xy^T - yx^T + Y)\right),\\
	&&&\sigma_{Y,y,z} \text{ is sos},\\
	&&&\bmat Y & y\\ y^T & 1 \emat \succeq 0.
\eaeql

Through straightforward algebra and matching coefficients, the first constraint (keeping in mind that $p$ is as in (\ref{Eq: Cubic Poly Form})) can be more explicitly written as:
\begin{align*}
b_i &= -e_i^TQy - \frac{1}{2} \Tr(H_iY), i = 1, \ldots, n,\\
- \gamma &= \frac{1}{6}\Tr(QY) + \frac{z}{3}.
\end{align*}

These constraints reflect that the coefficients of the linear terms and the scalar coefficient match on both sides; the cubic and quadratic coefficients are automatically the same. We can rewrite (\ref{Eq: new cubic sdp first term}) as
$$\sigma_{Y,y,z}(x) = \bvec x\\ 1\evec^T T(Y,y,z) \bvec x \\ 1 \evec,$$
where
$$T(Y,y,z) \defeq \bmat \sum_{i=1}^n y_iH_i + Q & \sum_{i=1}^n \Tr(H_iY)e_i+Qy \\ (\sum_{i=1}^n \Tr(H_iY)e_i+Qy)^T & z\emat.$$
The constraint in (\ref{Eq: new cubic sdp}) that $\sigma$ be sos is the same as the matrix $T$ being psd. Putting everything together, the problem in (\ref{Eq: new cubic sdp}) can be rewritten as the following SDP:\footnote{Recall that the data to this SDP is obtained from the representation of $p$ in the form of (\ref{Eq: Cubic Poly Form}).}

\begin{equation}\label{Eq: Small cubic SDP}
    \begin{aligned}
	& \underset{Y \in \Snn, y \in \Rn, z \in \R}{\inf}
	& & \frac{1}{6}\Tr(QY) + \frac{z}{3} \\
	& \text{subject to}
	&& \frac{1}{2}\Tr(H_iY) + e_i^TQy + b_i = 0, \forall i = 1, \ldots, n,\\
	&&& T(Y,y,z) \succeq 0,\\
	&&& \bmat Y & y\\ y^T & 1 \emat \succeq 0.
\eaeql

It is interesting to observe that the first constraint is a relaxation of the quadratic constraint which would impose $\grad p(y) = 0$, and that the constraint $T(Y,y,z) \succeq 0$ in particular implies $\Hess p(y) \succeq 0$. One can think of (\ref{Eq: Small cubic SDP}) as another SDP relaxation of (\ref{Eq: convex pop}) which is tight when $p$ has a second-order point.

\subsubsection{Combining the SDP in (\ref{Eq: Small cubic SDP}) with its Dual}
In this subsection, we write down an SDP (given in (\ref{Eq: complete cubic SDP})) whose optimal value can be related to the existence of second-order points of a cubic polynomial. To arrive at this SDP, we first take the dual of (\ref{Eq: Small cubic SDP}). It will turn out that the constraints in the dual follow a very similar structure to those in the primal, and that any feasible solution of the primal yields a feasible solution of the dual. We then combine the primal-dual pair of SDPs to arrive at a single SDP, which is the one in (\ref{Eq: complete cubic SDP}). To this end, let us write down the dual of (\ref{Eq: Small cubic SDP}):

\baeq
&\underset{R,S,r,s,\lambda,\sigma,\rho,\gamma}{\sup}
& & \gamma &&\\
& \text{subject to}
&& \frac{1}{6}\Tr(QY) + \frac{z}{3} - \gamma &=& \sum_{i=1}^n \lambda_i\left(\frac{1}{2}\Tr(H_iY)+e_i^TQy+b_i\right)\\
&&&&& + \Tr\left(\bmat Y & y \\ y^T & 1\emat \bmat R & r \\ r^T & \rho\emat \right)+ \Tr\left(T(Y,y,z) \bmat S & s \\ s^T & \sigma \emat\right), \forall (Y,y,z)\\
&&& \hspace{1.6cm} \bmat R & r \\ r^T & \rho\emat &\succeq& 0,\\
&&& \hspace{1.6cm} \bmat S & s \\ s^T & \sigma \emat &\succeq& 0,
\eaeq
where $R,S \in \Snn, r,s,\lambda \in \Rn,$ and $\sigma,\rho,\gamma \in \R$. The right-hand side of the first constraint simplifies to
$$b^T\lambda + \rho + \Tr(QS) +\Tr\left(\left(\sum_{i=1}^n (\frac{1}{2}\lambda_i+2s_i)H_i + R \right)Y\right) + \left(Q(\lambda+2s) + \sum_{i=1}^n \Tr(H_iS)e_i + 2r\right)^Ty + \sigma z.$$
After matching coefficients, the dual problem can be rewritten as
\baeq
    & \underset{R,S,r,s,\lambda,\rho}{\sup}
	&& -b^T\lambda - \rho - \Tr(QS)\\
	& \text{subject to}
	&& \sum_{i=1}^n (\frac{1}{2}\lambda_i+2s_i)H_i + R = \frac{1}{6}Q,\\
	&&& Q(\lambda+2s) + \sum_{i=1}^n \Tr(H_iS)e_i + 2r= 0,\\
	&&& \bmat R & r \\ r^T & \rho\emat \succeq 0,\\
	&&& \bmat S & s \\ s^T & \frac{1}{3} \emat \succeq 0,
\eaeq

Substituting $R$ and $r$ using the first two constraints into the first psd constraint and then multiplying by 6, we arrive at the problem

\baeq
    & \underset{S, s, \lambda, \rho}{\sup}
	&& -b^T\lambda - \rho - \Tr(QS)&\\
	& \text{subject to}
	&& \bmat \sum_{i=1}^n (-3\lambda_i - 12s_i)H_i + Q & Q(-3\lambda-6s) - 3\sum_{i=1}^n \Tr(H_iS)e_i \\ \left(Q(-3\lambda-6s) - 3\sum_{i=1}^n \Tr(H_iS)e_i \right)^T & 6\rho \emat &\succeq 0,\\
	&&& \hspace{11cm} \bmat S & s \\ s^T & \frac{1}{3}\emat &\succeq 0.
\eaeq

    Replacing $S$ with $\frac{1}{3}S$, $s$ with $-\frac{1}{3}s$, and $\rho$ with $\frac{1}{6}\rho$, we can reparameterize this problem and arrive at our final form for the dual of (\ref{Eq: Small cubic SDP}):
    
    \begin{equation}\label{Eq: small cubic SDP dual}
    \begin{aligned}
        & \underset{S, s, \lambda, \rho}{\sup}
    	&& -b^T\lambda - \frac{1}{6}\rho - \frac{1}{3}\Tr(QS)&\\
    	& \text{subject to}
    	&& \bmat \sum_{i=1}^n (4s_i - 3\lambda_i)H_i + Q & Q(2s - 3\lambda) - \sum_{i=1}^n \Tr(H_iS)e_i \\ \left(Q(2s-3\lambda) - \sum_{i=1}^n \Tr(H_iS)e_i\right)^T & \rho \emat &\succeq 0,\\
    	&&& \hspace{10cm} \bmat S & s \\ s^T & 1\emat &\succeq 0.
    \eaeql
    
    One can easily verify that if $(Y,y,z)$ is feasible to (\ref{Eq: Small cubic SDP}), then $(Y,y,y,z)$ is feasible to (\ref{Eq: small cubic SDP dual}). Replacing $(S,s,\lambda,\gamma)$ with $(Y,y,y,z)$ in (\ref{Eq: small cubic SDP dual}) gives an SDP whose constraints are the two psd constraints in (\ref{Eq: Small cubic SDP}) and whose objective function is $-b^Ty - \frac{1}{6}z - \frac{1}{3}\Tr(QY)$. We now create a new SDP, which has the same decision variables and constraints as (\ref{Eq: Small cubic SDP}), but whose objective function is the difference between the objective function of (\ref{Eq: Small cubic SDP}) and $-b^Ty - \frac{1}{6}z - \frac{1}{3}\Tr(QY)$. The optimal value of this new SDP is an upper bound on the duality gap of the primal-dual SDP pair (\ref{Eq: Small cubic SDP}) and (\ref{Eq: small cubic SDP dual}). If our cubic polynomial $p$ is written in the form (\ref{Eq: Cubic Poly Form}) and $$T(Y,y,z) = \bmat \sum_{i=1}^n y_iH_i + Q & \sum_{i=1}^n \Tr(H_iY)e_i+Qy \\ (\sum_{i=1}^n \Tr(H_iY)e_i+Qy)^T & z\emat$$
    as before, the new SDP we just described can be written as

    \begin{equation}\label{Eq: complete cubic SDP}
        \begin{aligned}
    	& \underset{Y \in \Snn, y \in \Rn, z \in \R}{\inf}
    	& & \frac{1}{2}\Tr(QY) + b^Ty + \frac{z}{2}\\
    	& \text{subject to}
    	&& \frac{1}{2}\Tr(H_iY)+e_i^TQy+b_i=0, \forall i = 1, \ldots, n,\\
    	&&& T(Y,y,z) \succeq 0,\\
    	&&& \bmat Y & y \\ y^T & 1 \emat \succeq 0.
    \eaeql
The following theorem relates the optimal value of this SDP to the existence of second-order points of $p$.

\begin{theorem}\label{Thm: Complete Cubic SDP}
    For a cubic polynomial $p$ given in the form (\ref{Eq: Cubic Poly Form}), consider the SDP in (\ref{Eq: complete cubic SDP}). For any feasible solution $(Y,y,z)$ to (\ref{Eq: complete cubic SDP}), the objective value of (\ref{Eq: complete cubic SDP}) is nonnegative. Furthermore, the optimal value of (\ref{Eq: complete cubic SDP}) is zero and is attained if and only if $p$ has a second-order point.
\end{theorem}
\begin{proof}
Suppose $(Y,y,z)$ is a feasible solution to (\ref{Eq: complete cubic SDP}). Note that $(Y,y,z)$ is feasible to (\ref{Eq: Small cubic SDP}) and $(Y,y,y,z)$ is feasible to (\ref{Eq: small cubic SDP dual}), and so
$$\frac{1}{2}\Tr(QY) + b^Ty + \frac{z}{2} = \frac{1}{6}\Tr(QY) + \frac{z}{3} - \left(-b^Ty - \frac{1}{6}z - \frac{1}{3}\Tr(QY)\right) \ge 0$$
by weak duality applied to (\ref{Eq: Small cubic SDP}) and (\ref{Eq: small cubic SDP dual}). Hence, the objective of (28) is nonnegative at any feasible solution.

Now suppose that $p$ has a second-order point $\xbar$. We claim that the triplet $$\left(\xbar\xbar^T, \xbar, \xbar^T\left(\sum_{i=1}^n \xbar_iH_i + Q\right)\xbar\right)$$ is feasible to (\ref{Eq: complete cubic SDP}) and achieves an objective value of zero. Indeed, the first constraint of (\ref{Eq: complete cubic SDP}) is satisfied because its left-hand side reduces to $\gxb$, which is zero. The third constraint is satisfied since the matrix $(\xbar, 1)(\xbar, 1)^T$ is clearly psd. The second constraint is satisfied since $T(\xbar\xbar^T, \xbar, \xbar^T(\sum_{i=1}^n \xbar_iH_i + Q)\xbar)$ can be written as
$$\bmat \sum_{i=1}^n \xbar_iH_i + Q & (\sum_{i=1}^n \xbar_iH_i+Q)\xbar \\ \xbar^T(\sum_{i=1}^n \xbar_iH_i+Q) & \xbar^T(\sum_{i=1}^n \xbar_iH_i + Q)\xbar\emat = \bmat (\sum_{i=1}^n \xbar_iH_i + Q)^{\frac{1}{2}} \\ \xbar^T(\sum_{i=1}^n \xbar_iH_i + Q)^{\frac{1}{2}}\emat \bmat (\sum_{i=1}^n \xbar_iH_i + Q)^{\frac{1}{2}} \\ \xbar^T(\sum_{i=1}^n \xbar_iH_i + Q)^{\frac{1}{2}}\emat^T.$$ The objective value at $(\xbar\xbar^T, \xbar, \xbar^T(\sum_{i=1}^n \xbar_iH_i + Q)\xbar)$ is
\begin{align*}
    &\frac{1}{2}\Tr(Q\xbar\xbar^T) + b^T\xbar + \frac{1}{2}\xbar^T\left(\sum_{i=1}^n \xbar_iH_i + Q\right)\xbar\\
    =& \frac{1}{2}\xbar^TQ\xbar - \left(\frac{1}{2} \sum_{i=1}^n \xbar_iH_i\xbar +Q\xbar\right)^T \xbar + \frac{1}{2}\xbar^T\left(\sum_{i=1}^n \xbar_iH_i + Q\right)\xbar\\
    =& 0.
\end{align*}
Since we have already shown that the objective function of (\ref{Eq: complete cubic SDP}) is nonnegative over its feasible set, it follows that when $p$ has a second-order point, the optimal value of (\ref{Eq: complete cubic SDP}) is zero and is attained.

To prove the converse, suppose the optimal value of (\ref{Eq: complete cubic SDP}) is zero and is attained. Let $(Y^*, y^*, z^*)$ be an optimal solution to (\ref{Eq: complete cubic SDP}). We will show that $y^*$ is a second-order point for $p$. Clearly $\Hess p(y^*)$ is psd, since $T(Y^*, y^*, z^*) \succeq 0$. To show that $\grad p(y^*) = 0$, let us start by letting $D \defeq Y^* - y^*y^{*T}$, and $d \defeq \sum_{i=1}^n \Tr(H_iD)e_i$. Note that
$$\frac{1}{2}\Tr(H_iy^*y^{*T}) + \frac{1}{2}\Tr(H_iD) + e_i^TQy^* + b_i = 0 \overset{(\ref{Eq: Cubic Poly Form})}{\Rightarrow} -2(\grad p(y^*))_i = \Tr(H_iD),$$
or equivalently $d = -2\grad p(y^*)$. In the remainder of the proof, we show that $d = 0$.

Since $\sum_{i=1}^n y_i^*H_iy^*$ is the vector whose $i$-th entry is $y^{*T}H_iy^*$, we have that 
\begin{equation}\label{Eq: Top right vector decomposition}
\sum_{i=1}^n \Tr(H_iY^*)e_i+Qy^*= \left(\sum_{i=1}^n y_i^*H_i + Q\right) y^* + d.
\end{equation} Then from the generalized\footnote{Here, $A^{+}$ refers to any pseudo-inverse of $A$, i.e. a matrix satisfying $AA^{+}A = A$.} Schur complement condition applied to $T(Y^*,y^*,z^*)$, we have
\begin{align*}
    z^* &\ge \left(\left(\sum_{i=1}^n y_i^*H_i + Q\right)y^* + d\right)^T \left(\sum_{i=1}^n y_i^*H_i + Q\right)^{+} \left(\left(\sum_{i=1}^n y_i^*H_i + Q\right)y^* + d\right)\\
    &= y^{*T}\left(\sum_{i=1}^n y_i^*H_i + Q\right)y^* + 2d^T\left(\sum_{i=1}^n y_i^*H_i + Q\right)^{+}\left(\sum_{i=1}^n y_i^*H_i + Q\right)y^* + d^T\left(\sum_{i=1}^n y_i^*H_i + Q\right)^{+}d.
\end{align*}
It is not difficult to verify that since $T(Y^*,y^*,z^*) \succeq 0$, we have $$\sum_{i=1}^n \Tr(H_iY^*)e_i+Qy^* \in \cols(\sum_{i=1}^n y_i^*H_i + Q),$$ and thus (\ref{Eq: Top right vector decomposition}) implies $d \in \cols(\sum_{i=1}^n y_i^*H_i + Q)$. Therefore, there exists a vector $v \in \Rn$ such that $d = (\sum_{i=1}^n y_i^*H_i + Q) v$. We then have
\begin{align*}
    d^T\left(\sum_{i=1}^n y_i^*H_i + Q\right)^{+}\left(\sum_{i=1}^n y_i^*H_i + Q\right) y^* &= v^T \left(\sum_{i=1}^n y^*_iH_i + Q\right) \left(\sum_{i=1}^n y^*_iH_i + Q\right)^{+} \left(\sum_{i=1}^n y^*_iH_i + Q\right) y^*\\
    &= v^T \left(\sum_{i=1}^n y_i^*H_i + Q\right) y^*\\
    &= d^T y^*.
\end{align*}
Now let $$\delta \defeq z^* - y^{*T}\left(\sum_{i=1}^n y_i^*H_i + Q\right)y^* - 2d^Ty^* - d^T\left(\sum_{i=1}^n y_i^*H_i + Q\right)^{+}d$$ and observe that $\delta \ge 0$. We can then write the objective value of (\ref{Eq: complete cubic SDP}) at $(Y^*,y^*,z^*)$ in terms of $D, d$, and $\delta$:

\begin{equation}\label{Eq: Complete SDP nonnegative}
\begin{aligned}
    \frac{1}{2}\Tr(QY^*) + b^Ty^* + \frac{1}{2}z^*
	&= \frac{1}{2}\left(y^{*T}Qy^* + \Tr(QD)\right) + \sum_{i=1}^n \left(-e_i^TQ y^*- \frac{1}{2}\Tr(H_iy^*y^{*T}) - \frac{1}{2}\Tr(H_iD)\right)y_i^*\\
	&+ \frac{1}{2}\left(y^{*T}\left(\sum_{i=1}^n y_i^*H_i + Q\right)y^* + 2d^Ty^* + d^T\left(\sum_{i=1}^n y_i^*H_i + Q\right)^{+}d+ \delta\right)\\
	&= \left(\frac{1}{2}-1+\frac{1}{2}\right)y^{*T}Qy^* + \left(-\frac{1}{2} + \frac{1}{2}\right)\sum_{i=1}^n y^{*T}y_i^*H_iy^*\\
	&+ \frac{1}{2}\Tr(QD) + \left(-\frac{1}{2}+1\right)\sum_{i=1}^n\Tr(H_iD)y_i^* + \frac{1}{2}d^T\left(\sum_{i=1}^n y_i^*H_i + Q\right)^{+}d + \frac{\delta}{2}\\
	&= \frac{1}{2}\Tr\left(\left(\sum_{i=1}^n y_i^*H_i + Q\right)D\right) + \frac{1}{2}d^T\left(\sum_{i=1}^n y_i^*H_i + Q\right)^{+}d + \frac{\delta}{2}\\
	& \ge 0,
\end{aligned}
\end{equation}
where in the last inequality we used the facts that $D \succeq 0$ and that the pseudo-inverse of a psd matrix is psd.
    
Since the left-hand side of the above equation is zero by assumption, and since all three terms on the right-hand side are nonnegative, it follows that $(\sum_{i=1}^n y_i^*H_i + Q)^+ d = 0$. As the null space of $(\sum_{i=1}^n y_i^*H_i + Q)^+$ is the same as the null space of $(\sum_{i=1}^n y_i^*H_i + Q)$, we have $(\sum_{i=1}^n y_i^*H_i + Q)d = 0$. However, because $d \in \cols(\sum_{i=1}^n y_i^*H_i + Q)$, it must be that $d=0$.

\end{proof}

\subsubsection{An Algorithm for Finding Local Minima}\label{SSSec: SOP and Algorithm}

Theorem \ref{Thm: Complete Cubic SDP} leads to the following characterization of second-order points of a cubic polynomial.

\begin{cor}\label{Cor: Complete Cubic SDP SOP}
    Let $p: \Rn \to \R$ be a cubic polynomial written in the form (\ref{Eq: Cubic Poly Form}). Then the set of its second-order points is equal to
    \beq\label{Eq: Complete Cubic SDP SOP}
    \begin{aligned}\{y \in \Rn\ |\ & \exists Y \in \Snn,z \in \R \text{ such that }\\
    &\frac{1}{2}\Tr(QY) + b^Ty + \frac{z}{2} = 0, \frac{1}{2}\Tr(H_iY)+e_i^TQy+b_i = 0, \forall i = 1, \ldots, n,\\
    &T(Y,y,z) \succeq 0, \bmat Y&y\\y^T&1\emat \succeq 0\}.
    \end{aligned}\eeq
\end{cor}
\begin{proof}
Recall from the proof of Theorem \ref{Thm: Complete Cubic SDP} that if $\xbar$ is a second-order point of $p$, then the triplet $(\xbar\xbar^T, \xbar, \xbar^T(\sum_{i=1}^n \xbar_iH_i + Q)\xbar)$ is feasible solution to (\ref{Eq: complete cubic SDP}) with objective value zero. Hence any second-order point belongs to (\ref{Eq: Complete Cubic SDP SOP}). Conversely, recall that if $(Y,y,z)$ is a feasible solution to (\ref{Eq: complete cubic SDP}) with objective value zero, then $y$ is a second-order point of $p$. Therefore any point in (\ref{Eq: Complete Cubic SDP SOP}) is a second-order point of $p$.
\end{proof}

In view of Theorem \ref{Thm: Closure}, we observe that if $p$ has a local minimum, the set in (\ref{Eq: Complete Cubic SDP SOP}) is a semidefinite representation of $\overline{LM_p}$. This observation gives rise to the following algorithm which tests if a cubic polynomial has a local minimum.

\begin{algorithm}[H]
	\caption{Algorithm for finding a local minimum of a cubic polynomial using a polynomial number of calls to E-SDP.}\label{Alg: Complete Cubic SDP}
	\begin{algorithmic}[1]
		\State {\bf Input:} A cubic polynomial $p: \Rn \to \R$ in the form (\ref{Eq: Cubic Poly Form})
		\State \texttt{{\bf TEST1}} test using E-SDP if (\ref{Eq: Complete Cubic SDP SOP}) is empty
		\State \quad \texttt{{\bf if}} YES
			\State \quad \quad \Return NO LOCAL MINIMUM
		\State \quad \texttt{{\bf if}} NO
			\State \quad\quad Find (via Lemma \ref{Lem: Relint recovery}) a point $x^*$ in the relative interior of (\ref{Eq: Complete Cubic SDP SOP})

        \mbox{}
		\State \texttt{{\bf TEST2}} test (via Theorem \ref{Thm: Checking Min Poly Time}) if $x^*$ is a local minimum
		\State \quad \texttt{{\bf if}} YES
			\State \quad \quad \Return $x^*$
		\State \quad \texttt{{\bf if}} NO
			\State \quad\quad \Return NO LOCAL MINIMUM
	\end{algorithmic}
\end{algorithm}

{\bf Complexity and correctness of Algorithm \ref{Alg: Complete Cubic SDP}.} By design, if $p$ has no local minimum, Algorithm \ref{Alg: Complete Cubic SDP} will return \texttt{NO LOCAL MINIMUM} since \texttt{TEST2} answers \texttt{NO} for every point. If $p$ has a local minimum, then $SO_p$ is nonempty. Since $SO_p$ is given by (\ref{Eq: Complete Cubic SDP SOP}) due to Corollary~\ref{Cor: Complete Cubic SDP SOP}, \texttt{TEST1} answers \texttt{YES}. Then, by Theorem \ref{Thm: Local Minima SOP}, any point in the relative interior of (\ref{Eq: Complete Cubic SDP SOP}) is a local minimum. Hence $x^*$ will pass \texttt{TEST2}. Note that this algorithm makes $2n+1$ calls to E-SDP, and then runs Algorithm~\ref{Alg: Check Local Min}.\footnote{In fact, the number of calls to E-SDP can be reduced to $2n$ if the very first call to E-SDP uses $x_1$ as the objective function.}

\begin{remark}\label{Rem: Find SLM}
{\bf Finding strict local minima.}
If we are specifically interested in searching for a strict local minimum of a cubic polynomial, we can simply check if the point $x^*$ returned by Algorithm~\ref{Alg: Complete Cubic SDP} satisfies $\Hess p(x^*) \succ 0$. If the answer is yes, we return $x^*$; if the answer is no, we declare that $p$ has no strict local minimum. Clearly, if a local minimum $x^*$ satisfies $\Hess p(x^*) \succ 0$, it must be a strict local minimum due to the SOSC. Furthermore, recall from Section \ref{SSec: Cubic Preliminaries} that if $p$ has a strict local minimum, then it has a unique local minimum, and thus that must be the output of Algorithm \ref{Alg: Complete Cubic SDP}.
\end{remark}

\section{Conclusions and Future Directions}\label{Sec: Conclusions}

In this paper, we considered the notions of (i) critical points, (ii) second-order points, (iii) local minima, and (iv) strict local minima for multivariate polynomials. For each type of point, and as a function of the degree of the polynomial, we studied the complexity of deciding (1) if a given point is of that type, and (2) if a polynomial has a point of that type. See Tables~\ref{Table: Complexity Checking} and \ref{Table: Complexity Existence} in Section~\ref{Sec: Introduction} for a summary of how our results complement prior literature. The majority of our work was dedicated to the case of cubic polynomials, where some new tractable cases were revealed based in part on connections with semidefinite programming. In this final section, we outline two future research directions which also have to do with cubic polynomials.


\subsection{Approximate Local Minima}\label{SSec: Approximate Local Minima}

In Sections~\ref{Sec: Complexity} and \ref{Sec: Finding Local Min}, we established polynomial-time equivalence of finding local minima and second-order points of cubic polynomials and some SDP feasibility problems (see Corollary~\ref{Cor: Complete Cubic SDP SOP}, Algorithm~\ref{Alg: Complete Cubic SDP}, Theorem~\ref{Thm: SOP SDPF}, Theorem~\ref{Thm: LM SDPSF}). Unless some well-known open problems around the complexity of SDP feasibility are resolved (see Section~\ref{Sec: Complexity}), one cannot expect to make claims about finding local minima of cubic polynomials in polynomial time in the Turing model of computation. Nonetheless, it is known that under some assumptions, one can solve semidefinite programs to arbitrary accuracy in polynomial time (see, e.g. \cite{ramana1997exact,alizadeh1995interior,vandenberghe1996semidefinite,porkolab1997complexity,nesterov1994interior,grotschel2012geometric}). It is therefore reasonable to ask if one can find local minima of cubic polynomials to arbitrary accuracy in polynomial time. This is a question we would like to study more rigorously in future work. We present a partial result in this direction in Theorem \ref{Thm: Eps local min} below.

Recall from Section \ref{SSec: SDP Approach} that our ability to find local minima of a cubic polynomial $p$ depended on our ability to minimize $p$ over its convexity region $CR_p$. We show next that we can find an $\epsilon$-minimizer of $p$ over $CR_p$ by approximately solving a semideifnite program.

\begin{theorem}\label{Thm: Eps local min}
    For a cubic polynomial $p$ given in the form (\ref{Eq: Cubic Poly Form}), consider the SDP in (\ref{Eq: complete cubic SDP}). If the objective value at a feasible point $(Y,y,z)$ is $\epsilon \ge 0$, then $p(y) \le p(x) + \frac{2}{3}\epsilon, \forall x \in CR_p$.
\end{theorem}
\begin{proof}
Consider a feasible solution $(Y,y,z)$ to (\ref{Eq: complete cubic SDP}). Let $\gamma^*$ be the infimum of $p$ over $CR_p$. Observe that
$$-\frac{1}{6}\Tr(QY) - \frac{z}{3} \le \gamma^*.$$
This is because the SDPs in (\ref{Eq: complete cubic SDP}) and (\ref{Eq: Small cubic SDP}) have the same constraints, and the optimal value of (\ref{Eq: Small cubic SDP}) is the negative of the optimal value of (\ref{Eq: new cubic sdp}), which by construction is a lower bound on $\gamma^*$. Similarly as in the proof of Theorem \ref{Thm: Complete Cubic SDP}, let $D \defeq Y - yy^T$, $d \defeq \sum_{i=1}^n \Tr(H_iD)e_i$, and
$$\delta \defeq z - y^T\left(\sum_{i=1}^n y_iH_i + Q\right)y - 2d^Ty - d^T\left(\sum_{i=1}^n y_iH_i + Q\right)^{+}d.$$ We can then write:
	\begin{align*}
	\frac{1}{6}\Tr(QY) + \frac{z}{3} &= \frac{1}{6}\Tr(QY) + \frac{z}{3} - \sum_{i=1}^n \left(\frac{1}{2} \Tr(H_iY) + e_i^TQy + b_i\right)y_i\\
	&= \frac{1}{6}\left(\Tr(Qyy^T) + \Tr(QD)\right)\\
	&+\frac{1}{3}\left(y^T\left(\sum_{i=1}^n y_iH_i + Q\right)y + 2d^Ty + d^T\left(\sum_{i=1}^n y_iH_i + Q\right)^{+}d + \delta \right)\\
	&- \frac{1}{2}\left(\Tr\left(\sum_{i=1}^n y_iH_iyy^T\right) + \Tr\left(\sum_{i=1}^n y_iH_iD\right)\right) - y^TQy - b^Ty\\
	&= -\frac{1}{6}\sum_{i=1}^n y^Ty_iH_iy - \frac{1}{2}y^TQy - b^Ty\\
	&+ \frac{1}{6}\Tr\left(\left(\sum_{i=1}^n y_iH_i + Q\right)D\right) + \frac{1}{3}\left(d^T\left(\sum_{i=1}^n y_iH_i + Q\right)^{+}d\right) + \frac{\delta}{3}\\
	&= -p(y) + \frac{1}{6}\Tr\left(\left(\sum_{i=1}^n y_iH_i + Q\right)D\right) + \frac{1}{3}\left(d^T\left(\sum_{i=1}^n y_iH_i + Q\right)^{+}d\right) + \frac{\delta}{3}\\
	&\le -p(y) + \frac{2}{3}\epsilon,
	\end{align*}
	where the first equality is due to the first constraint in (\ref{Eq: complete cubic SDP}), and the last inequality follows from the last equation of (\ref{Eq: Complete SDP nonnegative}) with $(Y^*,y^*,z^*)$ replaced by $(Y,y,z)$ and the fact that $\sum_{i=1}^n y_iH_i+Q$ and $D$ are both psd matrices. We therefore conclude that
	$$p(y) - \frac{2}{3}\epsilon \le -\frac{1}{6}\Tr(QY) - \frac{z}{3} \le \gamma^*.$$
	 We then have that $p(y) \le p(x) + \frac{2}{3}\epsilon, \forall x \in CR_p$ as desired.
	 
\end{proof}

\subsection{Unregularized Third-Order Newton Methods}\label{SSec: Cubic Newton}

We end our paper with an interesting application of the problem of finding a local minimum of a cubic polynomial. Recall that Newton's method for minimizing a twice-differentiable function proceeds by approximating the function with its second-order Taylor expansion at the current iterate, and then moving to a critical point\footnote{If the function to be minimized is convex, this critical point will be a global minimum of the quadratic approximation.} of this quadratic approximation. It is natural to ask whether one can lower the iteration complexity of Newton's method for three-times-differentiable functions by using third-order information. An immediate difficulty, however, is that the third-order Taylor expansion of a function around any point will not be bounded below (unless the coefficients of all its cubic terms are zero). In previous work (see, e.g. \cite{nesterov2019implementable}), authors have gotten around this issue by adding a regularization term to the third-order Taylor expansion. In future work, we aim to study an unregularized third-order Newton method which in each iteration moves to a local minimum of the third-order Taylor approximation by applying Algorithm \ref{Alg: Complete Cubic SDP}. We would like to explore the convergence properties of this algorithm and conditions under which the algorithm is well defined at every iteration.

As a first step, let us consider the univariate case. For a function $f: \R \to \R$, the iterations of (classical) Newton's method read
\begin{equation}\label{Eq: Second Order Iterates}
x_{k+1} = x_k - \frac{f'(x_k)}{f''(x_k)}.
\end{equation}
The update rule of a third-order Newton method, which in each iteration moves to the local minimum of the third-order Taylor approximation, is given by
\begin{equation}\label{Eq: Third Order Iterates}
x_{k+1} = x_k - \frac{f''(x_k) - \sqrt{f''(x_k)^2 - 2f'(x_k)f'''(x_k)}}{f'''(x_k)}.
\end{equation}
We have already observed that in some settings, these iterations can outperform the classical Newton iterations. For example, consider the univariate function
\begin{equation}\label{Eq: 363 Function}
f(x) = 20x\arctan(x) - 10 \log(1+x^2) + x^2,
\end{equation}
which is strongly convex and has a (unique) global minimum at $x = 0$, where $f(x) = 0$; see Figure~\ref{Fig: Newton Iterates}. The first three derivatives of this function are
$$f'(x) = 20\arctan(x) + 2x,$$
$$f''(x) = 2+\frac{20}{1+x^2},$$
$$f'''(x) = \frac{-40x}{(1+x^2)^2}.$$
One can show that the basin of attraction of the global minimum of $f$ under the classical Newton iterations in (\ref{Eq: Second Order Iterates}) is approximately $[-1.7121, 1.7121]$. Starting Newton's method with $|x_0| \ge 1.7122$ results in the iterates eventually oscillating between $\pm 13.4942$. In contrast, the iterates of our proposed third-order Newton method in (\ref{Eq: Third Order Iterates}) are globally convergent to the global minimum of $f$. The iterations of both methods starting at $x_0 = 1.5$ are compared in Table \ref{Tab: Newton Iterates} and Figure \ref{Fig: Newton Iterates}, showing faster convergence to the global minimum for the third-order approach.

\begin{table}[H]
	{\begin{tabular}{| c | c | c |}\hline
         $k$ & $x_k$ & $f(x_k)$\\\hline
         0 &  1.5 & 19.9473\\\hline
         1 & -.2327 & .5910\\\hline
         2 & -.0030 & 1.0014e-4\\\hline
         3 & -8.3227e-9 & 1.4546e-15\\\hline
         4 & 2.3490e-9 & 1.1587e-16\\\hline
    \end{tabular}
	\hspace{1cm}
	\begin{tabular}{| c | c | c |}\hline
         $k$ & $x_k$ & $f(x_k)$\\\hline
         0 &  1.5 & 19.9473\\\hline
         1 & -1.2786 & 15.1411\\\hline
         2 & .8795 & 7.7329\\\hline
         3 & -.3396 & 1.2477\\\hline
         4 & .0230 & .0058\\\hline
    \end{tabular}}
	\centering
	\caption{Iterations of the third-order Newton method (left) and the classical Newton method (right) on the function $f$ in (\ref{Eq: 363 Function}) starting at $x_0 = 1.5$.}
	\label{Tab: Newton Iterates}
\end{table}

\begin{figure}[H]
	\centering
	\includegraphics[height=.4\textheight,keepaspectratio]{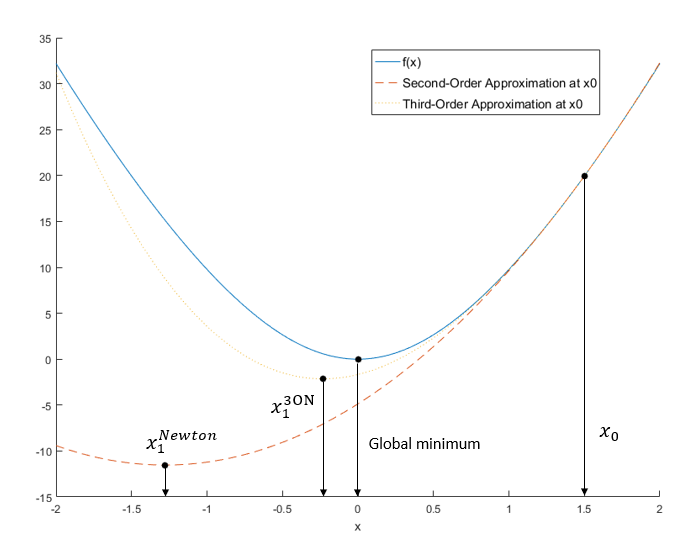}
	\caption{The plots of the function $f$ in (\ref{Eq: 363 Function}) and its second and third-order Taylor expansions around $x_0 = 1.5$. One can see that one iteration of the third-order Newton method in (\ref{Eq: Third Order Iterates}) brings us closer the global minimum of $f$ compared to one iteration of the Newton method in (\ref{Eq: Second Order Iterates}).} 
	\label{Fig: Newton Iterates}
\end{figure}

In addition to potential benefits regarding convergence, we have also observed that the behavior of the algorithm can be less sensitive to the initial condition when compared to Newton's method. As an example, we used Newton's method to find the critical points $\{1,-1,i,-i\}$ of $f(x) = x^5 - 5x$ on the complex plane, using the iterates (\ref{Eq: Second Order Iterates}), (\ref{Eq: Third Order Iterates}), and iterates given by
\begin{equation}\label{Eq: Third Order Iterates Max}
x_{k+1} = x_k - \frac{f''(x_k) + \sqrt{f''(x_k)^2 - 2f'(x_k)f'''(x_k)}}{f'''(x_k)},
\end{equation}
which can be interpreted as the iterates for moving to the local maximum of a third-order approximation of $f$. For each of the three iterations, the plots below demonstrate which initial conditions converge to the same critical point. As can be seen, sensitivity of Newton's method to the initial condition demonstrates fractal behavior, while the third-order iterates do not.

\begin{figure}[H]
	\centering
	\includegraphics[height=.25\textheight,keepaspectratio]{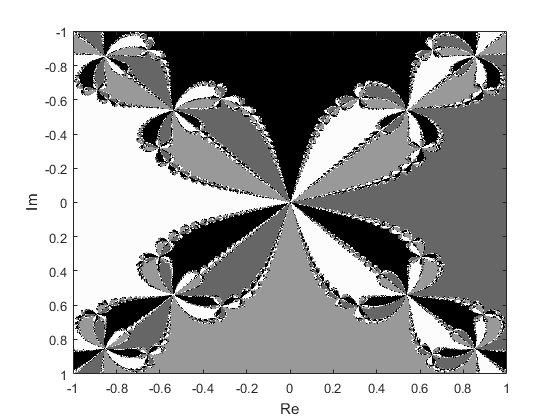}\\
	\includegraphics[height=.25\textheight,keepaspectratio]{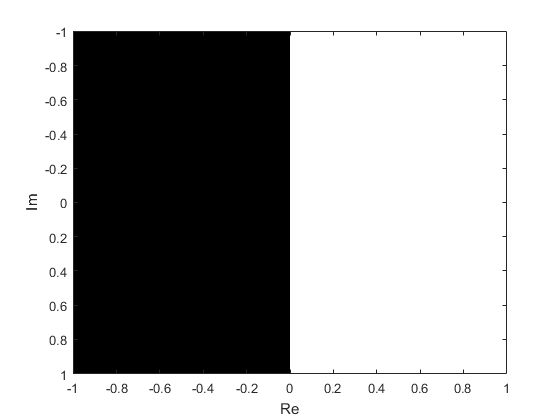}
	\includegraphics[height=.25\textheight,keepaspectratio]{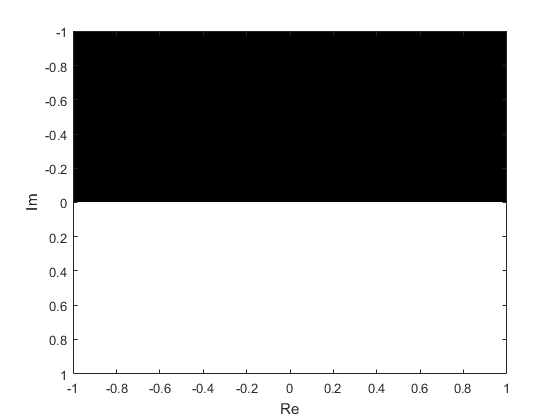}
	\caption{Sensitivity of the limits of the iterates (\ref{Eq: Second Order Iterates}), (\ref{Eq: Third Order Iterates}), and (\ref{Eq: Third Order Iterates Max}) respectively to initial conditions. Regions with the same color denote initial conditions which converge to the same critical point.} 
	\label{Fig: Newton Fractals}
\end{figure}

\bigskip
\noindent{\bf Acknowledgments:} We thank Bachir El Khadir for insightful comments and discussion. We also thank the Editorial Board of \emph{Advances in Mathematics} and an anonymous referee whose careful reading of the manuscript has improved the quality of our presentation.

\bibliographystyle{abbrv}
\bibliography{LM_refs}

\appendix

\end{document}